\documentclass[a4paper,11pt]{amsart}

\usepackage[margin=0.9in]{geometry}

\usepackage{amsmath}
\usepackage{amssymb}
\usepackage{xcolor}
\usepackage{bbm}
\usepackage{graphicx}
\usepackage{verbatim}

\usepackage[T1]{fontenc}
\usepackage[ansinew]{inputenc}
\usepackage[german, french,USenglish]{babel}

\makeatletter

\theoremstyle{plain}

\newtheorem{theorem}{Theorem}

\newtheorem{proposition}[theorem]{Proposition}

\theoremstyle{remark}

\newtheorem{remark}[theorem]{Remark}

\newcommand{\arginf}{\mathop{\mathrm{arg\,inf}}\displaylimits}

\newtheorem{assumption}{Assumption}

\def\Xset{\mathsf{S}}
\def\Aset{\mathsf{A}}
\def\dx{{d_\Xset}}
\def\da{{d_\Aset}}

\def\arginf{\operatornamewithlimits{arginf}}

\usepackage[colorinlistoftodos,bordercolor=orange,backgroundcolor=orange!20,linecolor=orange,textsize=scriptsize]{todonotes}

\numberwithin{equation}{section}

\def\JELname{{\bfseries JEL Classification}\enspace}
      \def\JEL#1{\par\addvspace\medskipamount{\rightskip=0pt plus1cm
      \def\and{\ifhmode\unskip\nobreak\fi\ $\cdot$
      }\noindent\JELname\ignorespaces#1\par}}
\begin{document}
\title[Primal-dual regression for MDP]{Primal-dual regression approach for Markov decision processes with general state and action space}
\author[D.~Belomestny]{Denis Belomestny$^{1}$}
\address{$^1$Faculty of Mathematics\\
Duisburg-Essen University\\
Thea-Leymann-Str.~9\\
D-45127 Essen\\
Germany}
\email{denis.belomestny@uni-due.de}

\author[J.~Schoenmakers]{John Schoenmakers$^{2}$}
\address{$^2$Weierstrass Institute for Applied Analysis and Stochastics \\
\mbox{Mohrenstr.~39} \\
10117 Berlin \\
Germany}
\email{schoenma@wias-berlin.de}


\date{}
\maketitle
\begin{abstract}
We develop a regression based primal-dual martingale approach for solving finite time horizon MDPs with general state and action space.
As a result, our method allows for
 the construction of tight upper and lower biased approximations of the value functions, and, provides tight approximations to the optimal policy.
 {In particular, we prove tight
 error bounds for the estimated duality gap featuring    polynomial dependence on the time horizon,
and sublinear dependence on the cardinality/dimension of the possibly infinite state and action space.}
 From a computational point of view the proposed method is efficient since, in contrast to usual duality-based methods for optimal control problems in the literature, the Monte Carlo procedures here involved do not require nested simulations.

\end{abstract}

\section{Introduction}

Markov decision processes (MDPs) provide a general framework for modeling sequential decision-making under uncertainty. A large number of practical problems from various areas such as economics, finance, and machine learning can be viewed as MDPs.
{For a classical reference
we refer to \cite{puterman2014markov}, and for MDPs with application to finance, see
\cite{BaRi}.}
The aim is usually to find an optimal policy that maximizes the expected accumulated rewards (or minimizes the expected accumulated costs).
{In principle, these Markov decision  problems can be  solved by a dynamic programming approach;}
 however, in practice, this approach suffers from the so-called ``curse of dimensionality'' and the ``curse of horizon''  meaning that the complexity of the program increases exponentially in the dimension of the problem (dimensions of the state and action spaces) and the horizon (at least for problems without discounting). While the curse of dimensionality is known to be unavoidable in the case of general continuous state/action spaces,   the possibility of beating the curse of the horizon remains an open issue.
\par
 A natural performance metric is given by the value function $V^{\pi}$ which is the expected total reward of the agent following $\pi$. Unfortunately, even a precise knowledge of $V^{\pi}$ does not provide reliable information on how far is the policy $\pi$ from the optimal one. To address this issue a popular quality measure is the \emph{regret} of the algorithm which is the difference between the total sum of rewards accumulated when following the optimal policy and the sum of rewards obtained when following the current policy $\pi$. In the setting of finite state- and action space MDPs there is a variety of regret bounds for popular RL algorithms like {\sf Q}-learning \cite{q_learning_efficient}, optimistic value iteration \cite{azar:2017}, and many others.
Unfortunately, regret bounds beyond the discrete setup are much less common in the literature. Even more crucial drawback of the regret-based comparison is that regret bounds are typically pessimistic and rely on the unknown quantities of the underlying MDP's. A simpler, but related, quantity is the \emph{suboptimality gap (policy error)} $\Delta_\pi(x) :=  V^\star(x)-V^\pi(x)$. Since we do not know \(V^\star\), the suboptimality gap can not be calculated directly. There is a vast amount of literature devoted to theoretical guarantees for $\Delta_\pi(x)$, see e.g. \cite{antos2007fitted}, \cite{szepesvari2010algorithms}, \cite{pires2016policy} and references therein. However, these bounds share the same drawbacks as the regret bounds. Moreover, known bounds do not apply to the general policy $\pi$ and depend heavily on the particular algorithm which produced it. For instance, in Approximate Policy Iteration ({\sf API}, \cite{Bertsekas+Tsitsiklis:1996}) all existing bounds  for \(\Delta_\pi(x)\) depend on the one-step  error induced by the approximation of the action-value function. This one-step error is difficult to quantify since it depends on the unknown smoothness properties of the action-value function. Similarly, in policy gradient methods (see e.g. \cite{sutton:book:2018}), there is always an approximation error due to the choice of the family of policies that can be hardly quantified.
Though the accuracy of a suboptimal policy is generally unknown, the lack of theoretical guarantees on a suboptimal policy can be potentially addressed by providing a dual bound, that is, an upper bound (or lower bound) on the optimal expected reward (or cost).
\par
The last decades have seen a high development of duality  approaches for optimal stopping and control problems, initiated by the works
of \cite{J_Rogers2002} and \cite{J_HK2004} in the context of pricing of American and Bermudan options.
Essentially, in the dual approach one
minimizes a certain {\em dual martingale representation} corresponding to the  problem under consideration
over a set of martingales or martingale type elements.
In general terms, the dual version of an optimal control problem
$
V_{0}^{\ast}=\sup_{\alpha}\mathbb{E}[R\left(  \alpha\right)  ]
$
for a reward $R$ depending on adapted policies $\alpha$
may be formulated as%
\[
V_{0}^{\ast}=\inf_{\text{martingales }M\left(  \boldsymbol{a}\right)
}\mathbb{E}[\sup_{\boldsymbol{a}\text{ in control space}}\left(R\left(
\boldsymbol{a}\right)  -M\left(  \boldsymbol{a}\right) \right) ].
\]
Thus, in the dual approach one seeks for optimal  martingales rather than optimal  policies.
For optimal stopping problems, \cite{J_AB2004} showed how to compute martingales using stopping rules via nested Monte Carlo simulations. In \cite{rogers_2007_pathwise},
the dual representation for optimal stopping (hence American options) was generalized to Markovian control problems.
Somewhat later \cite{J_BroSmiSun} presented a dual representation for quite general control problems in terms of the so-called information relaxation
and martingale penalties. On the other hand, the dual representation for optimal stopping was generalized to multiple stopping
in \cite{J_Schoen2010} and \cite{J_BenSchZan}.
As a numerical approach to \cite{rogers_2007_pathwise},  \cite{J_BelKolSch}
applied regression methods to solve Markov decision problems that can be seen, in a sense, as a generalization of \cite{J_AB2004}.
However, it should be noted that in the convergence analysis of  \cite{J_BelKolSch}, the primal value function estimates
showed exponential dependence on the time horizon, and the corresponding dual algorithm was based on nested simulations while its convergence  was not analyzed there.
 Generally speaking, to the best of our knowledge, all error bounds for the primal/dual value function estimates available in the literature so far show exponential dependence on the horizon at least in the case of finite horizon undiscounted optimal control problems, e.g. see also  \cite{Z}.
\par
In this paper, we propose a novel approach to constructing valid dual upper bounds on the optimal value function via simulations and pseudo regression in the case of finite horizon MDPs with general (possibly continuous) state and action spaces. {This approach
includes the construction of primal value functions via a backwardly structured pseudo regression procedure based on a
properly chosen reference distribution (measure).
We thus avoid the delicate problem of inverting the  empirical covariance matrices. Note that in the context of optimal stopping, a similar primal procedure was proposed in \cite{BayRedSch}, though with accuracy estimates exploding with the number of exercise dates or time horizon.
As for the dual part of our algorithm,  we avoid nested Monte Carlo simulation that were used in many dual-type methods proposed in the literature so far, see for instance the path-wise optimization approach for MDPs in \cite{desai2012bounds} and
\cite{brown2022information} for an overview.
Instead, for constructing the martingale elements  we propose to combine a pointwise pseudo regression approach
with a suitable interpolation method such that the
martingale property is preserved.
Furthermore, we provide a rigorous convergence analysis showing that the  error of approximating the true value function via estimated dual value function (duality gap) depends at most polynomially on the time horizon. Moreover, we show that the stochastic part of the error depends sublinearly on the dimension (or cardinality in the finite case) of the state and action spaces. Let us also mention  \cite{zhu2017solving} for another approach to avoid nested simulations when estimating the conditional expectations, hence the martingale elements, inside the dual representation. However, \cite{zhu2017solving} left
the issue of bounding the duality gap in terms of the error bounds on the primal value functions as an open problem.  In this respect, we have solved this problem within the context
of the algorithm proposed in this paper. }
\par
The paper is organized as follows. The basic setup of the Markov Decision Process and the well-known representations
for its maximal expected reward is given in Section~\ref{setup}.
Section~\ref{sec:dualrep} recalls the dual representation for an MDP
from the literature. The primal pseudo regression algorithm
for the value functions is described in Section~\ref{ValAppr},
whereas the dual regression algorithm is presented in
Section~\ref{dualalg}. Section~\ref{sec:conv-value} and Section~\ref{convup} are dedicated to the convergence analysis of the primal and
dual algorithm, respectively. Appendix~A introduces some auxiliary notions
needed to formulate an auxiliary result in Appendix~B
stemming from the theory of empirical processes.

\section{Setup and basic properties of the Markov Decision Process}\label{setup}

We consider the discrete time finite horizon Markov Decision Process (MDP), given by the tuple
\begin{equation*}
\mathcal{M}=(\mathsf{S},\mathsf{A},(P_{h})_{h\in]H]},(R_{h})_{h\in[H[},F,H),
\end{equation*}
made up by the following items:
\begin{itemize}
\item a measurable state space $(\mathsf{S},\mathcal{S})$ which may be finite or infinite;
\item a measurable action space $(\mathsf{A},\mathcal{A})$ which may be finite or infinite;
\item an integer $H$ which defines the horizon of the problem;
\item for each $h\in ]H],$ with $]H]:=\{1,\ldots,H\}$\footnote{We further write $[H]:=\{0,1,\ldots,H\}$ etc.}, a time dependent transition function $P_{h}:$ $\mathsf{S}\times\mathsf{A}\to\mathcal{P}(\mathsf{S})$
where $\mathcal{P}(\mathsf{S})$ is the space of probability measures on
$(\mathsf{S},\mathcal{S})$;
\item a time dependent reward function $R_{h}:$ $\mathsf{S}\times\mathsf{A}\to\mathbb{R},$
where $R_{h}(x,a)$ is the immediate reward associated with taking action
$a\in\mathsf{A}$ in state $x\in\mathsf{S}$ at time step $h\in[H[$;
\item a terminal reward $F:$ $\mathsf{S}\to\mathbb{R}$.
\end{itemize}
Introduce a filtered probability space $\mathfrak{S}:=\bigl(\Omega
,\mathcal{F},(\mathcal{F}_{t})_{t\in\lbrack H]},\mathbb{P}\bigr)$ with
\begin{equation}\label{omegaset}
\Omega:=\left(  \mathsf{S}\times\mathsf{A}\right)  ^{[H]}%
,\quad\mathcal{F}:=\left(  \mathcal{S}\otimes\mathcal{A}\right)
^{\otimes\lbrack H]},\quad(\mathcal{F}_{t})_{t\in\lbrack H]}:=
((\mathcal{S}\otimes\mathcal{A})^{\otimes t})_{t\in\lbrack
H]}.\footnote{In order to avoid irrelevant measure theoretic technicalities it
is assumed that our probability space is supported by discrete time processes,
rather than Wiener processes for instance. Nonetheless, it is possible to
involve larger probability spaces without essentially affecting the results in
this paper.}%
\end{equation}
For a fixed policy $\boldsymbol{\pi}=(\pi_{0},\ldots,\pi_{H-1})$ with $\pi
_{t}:$ $\mathsf{S}\rightarrow\mathcal{P}(\Aset),$ we consider an adapted
controlled process $\bigl(S_{t},A_{t}\bigr)_{t=h,\ldots,H}$ on $\mathfrak{S}$
satisfying $S_{0}\in\mathsf{S},$ $A_{0}\sim\pi_{0}(S_0),$ and
\[
S_{t+1}\sim P_{t+1}(\left.  \cdot\right\vert S_{t},A_{t}),\quad A_{t}\sim
\pi_{t}(S_{t}),\quad t=0,\ldots,H-1.
\]
The expected reward of this  MDP due to the chosen policy $\boldsymbol{\pi}$ is given by
\[
V_{0}^{\boldsymbol{\pi}}(x):=\mathbb{E}_{\boldsymbol{\pi},x}\left[
\sum_{t=0}^{H-1}R_{t}(S_{t},A_{t})+F(S_{H})\right], \quad x\in \mathsf{S}
\]
where \(\mathbb{E}_{\boldsymbol{\pi},x}\) stands for expectation induced by the policy \(\boldsymbol{\pi}\) and transition kernels \(P_{t},\) \(t\in [H],\)  conditional on the event \(S_0=x.\)
The
goal of the Markov decision problem is to determine the maximal expected
reward:
\begin{equation}
V_{0}^{\boldsymbol{\star}}:=\sup_{\boldsymbol{\pi}}\mathbb{E}%
_{\boldsymbol{\pi},x}\left[  \sum_{t=0}^{H-1}R_{t}(S_{t},A_{t})+F(S_{H})\right]
=\sup_{\boldsymbol{\pi}}V_{0}^{\boldsymbol{\pi}}(x_{0}).\label{Vopt}%
\end{equation}
Let us introduce for a generic time $h\in\left[H\right],$ the value
function due to the policy $\boldsymbol{\pi},$
\begin{align*}
V_{h}^{\boldsymbol{\pi}}(x)  & :=\mathbb{E}%
_{\boldsymbol{\pi},x}\left[
\left.  \sum_{t=h}^{H-1}R_{t}(S_{t},A_{t})+F(S_{H})\right\vert S_{h}=x\right]
,\text{ \ \ }x\in\mathsf{S}.
\end{align*}
Furthermore, let
\begin{equation}
\label{eq:optim_vh}
V_{h}^{\boldsymbol{\star}}(x)   :=\sup_{\boldsymbol{\pi}}%
V_{h}^{\boldsymbol{\pi}}(x)
\end{equation}
be the optimal value function at $h\in\left[  H\right]$. It is well known that under weak conditions,
there exists an optimal policy  solving \eqref{eq:optim_vh} which depends on $S_{t}$ in a deterministic
way. In this case, we shall write $\boldsymbol{\pi}^{\boldsymbol{\star}}%
=(\pi_{t}^{\boldsymbol{\star}}(S_{t}))$ for some mappings $\pi_{t}^{\star}:$
$\mathsf{S}\rightarrow\Aset$. One has
the following result, see \cite{puterman2014markov}.

\begin{theorem}\label{Bellman}
Let  $x\in\mathsf{S}$ be fixed. It holds  $V_{H}%
^{\boldsymbol{\star}}(x)=F(x),$ and
\begin{equation}
V_{h}^{\boldsymbol{\star}}(x)=\sup_{a\in A}\left(  R_{h}(x,a)+\mathbb{E}%
_{S_{h+1}\sim P_{h+1}(\cdot|x,a)}\left[  V_{h+1}^{\boldsymbol{\star}}%
(S_{h+1})\right]  \right)  ,\quad h=H-1,\ldots,0.\label{eq:bellman-q}%
\end{equation}
Furthermore, if $R_{h}$ is continuous and the action space is compact, the
supremum in (\ref{eq:bellman-q}) is attained at some deterministic optimal action
$a^{\boldsymbol{\star}}=\pi_{h}^{\boldsymbol{\star}}(x)$.
\end{theorem}
Let us further introduce recursively $Q_{H}^{\star}(x,a)=F(x),$ and
\[
Q_{h}^{\star}(x,a):= R_{h}(x,a)+\mathbb{E}_{S_{h+1}\sim P_{h+1}%
(\cdot|x,a)}\left[  \sup_{a^{\prime}\in A}Q_{h+1}^{\boldsymbol{\star}}%
(S_{h+1},a^{\prime})\right], \quad h=H-1,\ldots,0.
\]
Then $Q_{h}^{\star}(x,a)$ is called the \textit{optimal state-action} function ($Q$-function) and one
thus has%
\[
V_{h}^{\boldsymbol{\star}}(x)=\sup_{a\in A}Q_{h}^{\star}(x,a),\text{ \ \ }%
\pi_{h}^{\boldsymbol{\star}}(x)\in\arg\max_{a\in\mathsf{A}}Q_{h}^{\star
}(x,a),\text{ \ \ for}\quad h\in\lbrack H].
\]
Finally, note that the optimal value function $V^{\star}$ satisfies due to Theorem~\ref{Bellman},
\[
V_{h}^{\star}(x)=T_{h}V_{h+1}^{\star}(x),\text{ \ \ }h\in\lbrack H[,
\]
where $T_{h}V(x):=\sup_{a\in A}\left(  R_{h}(x,a)+P_{h+1}^{a}V(x)\right)  $
with $P_{h+1}^{a}V(x):=\mathbb{E}_{S_{h+1}\sim P_{h+1}(\cdot
|x,a)}\left[  V(S_{h+1})\right]  .$

\section{Dual representation }\label{sec:dualrep}
Let us denote by $a_{<t}$ the deterministic vector of actions $a_{<t}%
=(a_{0},\ldots,a_{t-1})\in\mathsf{A}^{t},$ similarly $a_{\leq t}$ etc., and
denote with $S_{t}\equiv(S_{t}(a_{<t}))_{t\in\{0,\ldots,H\}}$ the process
defined (in distribution) via
\[
S_{0}=x,\quad S_{t+1}\equiv S_{t+1}(a_{< t+1})\sim P_{t+1}(\cdot
|S_{t},a_{t}),\quad t=0,\ldots,H-1.
\]
Let us also denote by $\Xi$ the class of $H$-tuples $\boldsymbol{\xi}=(\xi
_{t}(\cdot,\cdot),$ $t\in]H])$ consisting of $\mathcal{A}^{\otimes t}%
\times\mathcal{F}_{t}$ measurable random variables
\begin{align*}
\xi_{t} &  :(a_{<t},\omega)\in\mathsf{A}^{t}\mathsf{\times\Omega\rightarrow
}\mathbb{R}
\end{align*}
 satisfying
 \begin{align*}
\mathbb{E}\left[  \left.  \xi_{t}(a_{<t},\omega)\right\vert \mathcal{F}%
_{t-1}\right]   &  =0,\quad\text{for all }(a_{<t})\in\mathsf{A}^{t},\quad
t\in\{1,\ldots,H\}.
\end{align*}

The next duality theorem, essentially due to \cite{rogers_2007_pathwise}, may
be seen as a generalization of the dual representation theorem for optimal
stopping, developed independently in \cite{J_Rogers2002} and \cite{J_HK2004},
to Markov decision processes. For a more general dual representations in terms
of information relaxation, see \cite{J_BroSmiSun}. Let us further mention dual
representations in the context of multiple stopping developed in
\cite{J_Schoen2010}, \cite{BenSchZan}, and applications to flexible caps
studied in \cite{J_BalMahSch}.

\begin{theorem}
\label{Rog}

The following statements hold.

\begin{description}

\item[(i)] For any $\boldsymbol{\xi}\in\Xi$ and any \(x\in \mathsf{S}\) we have \(V_{0}^{\mathrm{up}}(x;\boldsymbol{\xi})\geq V_{0}^{\star}(x)\) with
\begin{multline}
V_{0}^{\mathrm{up}}(x;\boldsymbol{\xi}):=\label{uppb}
\mathbb{E}_{\boldsymbol{\pi},x}\left[  \sup_{a_{\geq0}\in\mathsf{A}^{H}}\left(  \sum_{t=0}%
^{H-1}\left(  R_{t}(S_{t}(a_{<t}),a_{t})-\xi_{t+1}(a_{< t+1})\right)
+F(S_{H}(a_{<H}))\right)  \right]
\end{multline}
where as usual we suppress the dependence on $\omega$ for notational
simplicity. Hence $V_{0}^{\mathrm{up}}(x;\boldsymbol{\xi})$ is an upper (upper biased) bound
for $V_{0}^{\star}(x).$

\item[(ii)] If we set $\boldsymbol{\xi}^{\star}=(\xi_{t}^{\star},\,t\in\lbrack
H])\in\Xi$ with
\begin{equation}
\xi_{t+1}^{\star}(a_{< t+1}):= V_{t+1}^{\star}(S_{t+1}(a_{<
t+1}))-\mathbb{E}_{S_{t+1}^{\prime}\sim P_{t+1}(\cdot|S_{t}(a_{<t}),a_t)}\left[
V_{t+1}^{\star}(S_{t+1}^{\prime})\right]\label{eq:xistar}%
\end{equation}
for \(t=0,\ldots
,H-1,\) then, almost surely,
\begin{equation}
V_{0}^{\star}(x_{0})=\sup_{a_{\geq0}\in\mathsf{A}^{H}}\left(  \sum_{t=0}%
^{H-1}\left(  R_{t}(S_{t}(a_{<t}),a_{t})-\xi_{t+1}^{\star}(a_{< t+1})\right)
+F(S_{H}(a_{<H}))\right)  .\label{strong}%
\end{equation}

\end{description}
\end{theorem}

\begin{remark}
In Theorem~\ref{Rog} and further below, supremum should be interpreted as  essential supremum
in case it concerns the supremum over an uncountable family of random variables.
\end{remark}
In principle, Theorem~\ref{Rog} may be inferred from
\cite{rogers_2007_pathwise} or \cite{J_BroSmiSun}. Nonetheless, also for the
convenience of the reader, we here give a concise proof in terms of the
present notation and terminology.

\medskip

\begin{proof}
(i) Since for any $\boldsymbol{\xi}\in\Xi$ and policy $\boldsymbol{\pi}$ in
(\ref{Vopt}) one has that%
\[
\mathbb{E}%
_{\boldsymbol{\pi},x}\left[  \xi_{t+1}(A_{\leq t})\right]
=\mathbb{E}%
_{\boldsymbol{\pi},x}\mathbb{E}_{\boldsymbol{\pi}}\left[  \left.
\xi_{t+1}(A_{\leq t})\right\vert \mathcal{F}_{t}\right]  =0,
\]
for $t=h,...,H-1,$ it follows that%
\[
V_{0}^{\boldsymbol{\star}}(x)=\sup_{\boldsymbol{\pi}}\mathbb{E}%
_{\boldsymbol{\pi},x}\left[  \sum_{t=0}^{H-1}\left(  R_{t}(S_{t}(A_{<t}%
),A_{t})-\xi_{t+1}(A_{\leq t})\right)  +F(S_{H}(A_{<H}))\right]  ,
\]
from which (\ref{uppb}) follows immediately.

(ii) We may write for any $a_{\geq 0}\in\mathsf{A}^{H}$,%
\begin{multline*}
\sum_{t=0}^{H-1}\left(  R_{t}(S_{t}\left(  a_{<t}\right)  ,a_{t})-\xi
_{t+1}^{\star}(a_{\leq t})\right)  +F(S_{H}(a_{<H}))
\\
= \sum_{t=0}^{H-1}R_{t}(S_{t}\left(  a_{<t}\right)  ,a_{t})-\sum_{t=0}
^{H-1}V_{t+1}^{\star}(S_{t+1}(a_{\leq t}))
\\
 +\sum_{t=0}^{H-1}\mathbb{E}_{S_{t+1}^{\prime}\sim P_{t+1}(\cdot
|S_{t}(a_{<t}),a_t)}\left[  V_{t+1}^{\star}(S_{t+1}^{\prime})\right]
+F(S_{H}(a_{<H})).
\end{multline*}
Hence
\begin{equation*}
\sum_{t=0}^{H-1}\left(  R_{t}(S_{t}\left(  a_{<t}\right)  ,a_{t})-\xi
_{t+1}^{\star}(a_{\leq t})\right)  +F(S_{H}(a_{<H}))
  = V_{0}^{\star}(x)+\Delta(x),
\end{equation*}
with
\begin{multline*}
\Delta(x):=F(S_{H}(a_{<H}))-V_{H}^{\star}(S_{H}(a_{<H}))+
  \\
\sum_{t=0}^{H-1}\left(R_{t}(S_{t}\left(
a_{<t}\right)  ,a_{t})+\mathbb{E}_{S_{t+1}^{\prime}\sim P_{t+1}(\cdot
|S_{t},a_{t})}\left[  V_{t+1}^{\star}(S_{t+1}^{\prime})\right]   -V_{t}^{\star}(S_{t}(a_{<t})) \right)    \leq 0,
\end{multline*}
where the latter inequality follows from the Bellman principle, see
Theorem~\ref{Bellman}. The statement (\ref{strong}) now follows by taking the
 supremum over $a_{\geq0}\in\mathsf{A}^{H}$ on the left-hand-side,
applying (\ref{uppb}) and using the sandwich property.
\end{proof}

\section{Primal regression algorithm for the value function}\label{ValAppr}
In Section~\ref{dualalg} we will describe regression based martingale methods
for computing dual upper bounds based on Theorem~\ref{Rog}. However, these
methods require as an input a sequence of (approximate) value functions $V_{h},$
$h\in\lbrack H]$. Below we describe a regression-based  algorithm for
approximating the value functions $V_{h}^{\star},$ $h\in\lbrack H]$,
backwardly in time. In fact, unlike the usual regression, the proposed algorithm
is based on a kind of \textquotedblleft pseudo\textquotedblright\ or
\textquotedblleft quasi\textquotedblright\ regression procedure with respect
to some reference measure $\mu_h$ which is assumed to be such that
$P_{h}(\cdot|x,a)$ is absolutely continuous w.r.t. $\mu_h$ for any $h\in]H],$
$x\in\mathsf{S}$ and $a\in\mathsf{A}$. Furthermore, we consider a vector of
basis functions%
\[
\boldsymbol{\gamma}_{K}:=(\gamma_{1},\ldots,\gamma_{K})^{\top},\text{
\ \ }\gamma_{k}:\mathsf{S}\rightarrow\mathbb{R}\text{, \ \ }k=1,\ldots,K,
\]
such that the matrix%
\[
\Sigma\equiv\Sigma_{h,K}:=\mathbb{E}_{X\sim\mu_h}\left[  \boldsymbol{\gamma}_{K}%
(X)\boldsymbol{\gamma}_{K}^{\top}(X)\right]
\]
is analytically known and invertible. This basically means that the choice of basis functions is adapted to the choice of the reference measure. For example, if \(\mu_h\) is Gaussian one can chose basis functions to be polynomials or trigonometric polynomials.
The algorithm reads then as follows.
At $h=H$ we set $V_{H,N}(x)=V_{H}^{\star}(x)=F(x).$ Suppose that for some
$h\in\lbrack H[,$ the approximations $V_{t,N}$ of $V_{t}^{\star},$
$h+1\leq
t\leq H,$
are already obtained. We now approximate $V_{h}^{\star}$   via
simulating independent  random variables $X_{i}\equiv X_{i}^h\sim\mu_h,$ $Y_{i}^{a}\sim P_{h+1}(\cdot|X_{i},a),$
$a\in\mathsf{A}$, $i=1,\ldots,N,$ and setting%
\begin{align}
V_{h,N}(x)  & =T_{h,N}V_{h+1,N}(x):=\sup_{a\in\mathsf{A}}(R_{h}(x,a)+\widetilde{P}_{h+1,N}%
^{a}V_{h+1,N}(x))\label{lowconst}
\end{align}
where
\begin{equation}
\widetilde{P}_{h+1,N}^{a}V(x):=\mathcal{T}_{\widetilde{L}_{h+1}}[\beta
_{N,a}^{\top}\boldsymbol{\gamma}_{K}](x):=\max\bigl(-\widetilde{L}_{h+1}%
,\min\bigl(\widetilde{L}_{h+1},\beta_{N,a}^{\top}\boldsymbol{\gamma}%
_{K}(x)\bigr)\bigr)\label{7+}
\end{equation}
with $\widetilde{L}_{h+1}$ being a positive constant depending on $h,$ which
will be defined later, and
\begin{equation}
\beta_{N,a}:=\frac{1}{N}\sum_{i=1}^{N}U_{i}^{a},\quad U_{i}^{a}:=Z_{i}%
^{a}\Sigma^{-1}\boldsymbol{\gamma}_{K}(X_{i}),\quad Z_{i}^{a}:=V(Y_{i}%
^{a}),\quad i=1,\ldots,N.\label{7++}
\end{equation}
Note that $\beta_{a}:=\mathbb{E}\left[  \beta_{N,a}\right]  =\mathbb{E}\left[
V(Y_{1}^{a})\Sigma^{-1}\boldsymbol{\gamma}_{K}(X_{1})\right]  $ solves the
minimization problem%
\[
\inf_{\beta\in\mathbb{R}^{K}}\mathbb{E}\left[  \left(  V(Y_{1}^{a}%
)-\beta^{\top}\boldsymbol{\gamma}_{K}(X_{1})\right)  ^{2}\right]  .
\]
Thus, the quantity $\widetilde{P}_{h+1,N}^{a}V_{h+1,N}(x)$ aims to approximate
the conditional expectation
\[
x\rightarrow\mathbb{E}_{S^{\prime}\sim P_{h+1}(\cdot|x,a)}\left[
V_{h+1,N}(S^{\prime})\right]  ,\quad a\in\mathsf{A}.
\]
{The use of clipping at level $\widetilde{L}_{h+1}$ is done to avoid large
values of $\beta_{N,a}^{\top}\boldsymbol{\gamma}_{K}(x).$} After $H$ steps of
the above procedure we obtain the estimates $V_{H,N},\ldots,V_{0,N}%
.$\footnote{Actually, for computing $V_{0}(x_{0})$ we may replace the above
procedure by a standard Monte Carlo simulation when going from $V_{1}$ to
$V_{0}$.}

\section{Dual regression algorithm}

\label{dualalg} In this section we outline how to construct an upper biased
estimate based on Theorem~\ref{Rog} from a given sequence of approximations
$V_{t},$ $t\in[H]$  obtained, for example, as described in
Section~\ref{ValAppr}.
\par
Theorem~\ref{Rog}-(ii) implies that we can restrict our attention to processes
$\boldsymbol{\xi}=(\xi_t)_{t\in [H]}$, where the $t+1$ component of \(\boldsymbol{\xi}\) is of the form
\begin{equation}
\xi_{t+1}(a_{\leq t})=m(S_{t+1}(a_{\leq t});S_{t}(a_{<t}),a_{t})
\label{repksi}%
\end{equation}
for a deterministic real valued function $m(\cdot;x,a)$ satisfying
\begin{equation}
\int m(y;x,a)P_{t+1}(dy|x,a)=0, \label{eq:m-zeromean}%
\end{equation}
for all $(x,a)\in\mathsf{S}\times\mathsf{A}$. Note that the condition
(\ref{eq:m-zeromean}) is time dependent. We shall denote by $\mathcal{M}%
_{t+1,x,a}$ the set of \textquotedblleft martingale\textquotedblright%
\ functions $m$ on \(\mathsf{S}\) that satisfy (\ref{eq:m-zeromean}) for time $t+1$, a state $x$,
and a control $a.$ In this section, we develop an algorithm  approximating $\boldsymbol{\xi
}^{\star}$ via
regression of $V_{t+1}$ on a properly chosen finite dimensional subspace of
$\mathcal{M}_{t+1,x,a}.$ The idea of approximating $\boldsymbol{\xi}^{\star}$
via regression can be explained as follows. Equation (\ref{eq:xistar}) and
(\ref{repksi}) imply that, for a particular $t\in\lbrack H[,$ the component
$\xi_{t+1}^{\star}(a_{\leq t})$ of the random vector $\boldsymbol{\xi}^{\star
}$ is given by $\xi_{t+1}^{\star}(a_{\leq t})=m_{t+1}^{\star}(S_{t+1}(a_{\leq
t});S_{t}(a_{<t}),a_{t}),$ where, for each $(x,a)\in\mathsf{S}\times
\mathsf{A}$, $m_{t+1}^{\star}(\cdot;x,a)$ solves the optimization problem
\begin{multline}
\arginf_{m\in\mathcal{M}_{t+1,x,a}}\mathbb{E}_{S_{t+1}^{\prime}\sim P_{t+1}%
(\cdot|x,a)}\left[  \left(  V_{t+1}^{\star}(S_{t+1}^{\prime})-m(S_{t+1}%
^{\prime};x,a)\right)  ^{2}\right] =
\\
\arginf_{m\in\mathcal{M}_{t+1,x,a}}\mathrm{Var}_{S_{t+1}^{\prime}\sim P_{t+1}%
(\cdot|x,a)}\left[V_{t+1}^{\star}(S_{t+1}^{\prime})-m(S_{t+1}^{\prime};x,a)\right].
\label{eq:optim-dual-step}%
\end{multline}
By generating a sample $Y_{1}^{x,a},\ldots,Y_{N}^{x,a}$ from $P_{t+1}%
(\cdot|x,a)$ we readily obtain a computable approximation of $m_{t+1}^{\star
}(\cdot;x,a),$ that is, the solution of (\ref{eq:optim-dual-step}), by
\begin{equation}
\arginf_{m\in\mathcal{M}_{t+1,x,a}^{\prime}}\left\{  \frac{1}{N}\sum_{i=1}%
^{N}\left(  V_{t+1}(Y_{i}^{x,a})-m(Y_{i}^{x,a})\right)  ^{2}\right\}  ,
\label{eq:optim-dual-m}%
\end{equation}
where $\mathcal{M}_{t+1,x,a}^{\prime}$ is some \textquotedblleft large
enough\textquotedblright\ finite-dimensional subset of $\mathcal{M}%
_{t+1,x,a}.$

Let us now discuss possible constructions of the martingale functions $m$
satisfying \eqref{eq:m-zeromean}. Assume that $\mathsf{S}\subseteq
\mathbb{R}^{d}$ and that the conditional distribution $P_{t+1}(\cdot|x,a)$
possesses a smooth density $p_{t+1}(\cdot|x,a)$ with respect to the Lebesgue
measure on $\mathbb{R}^{d}.$ Furthermore, assume that $p_{t+1}(\cdot|x,a)$ doesn't vanish
on any compact set in $\mathbb{R}^{d},$ and that $p_{t+1}(y|x,a)$
$\rightarrow$ $0$ for $|y|$ $\rightarrow$ $\infty.$ Now consider, for any fixed $(x,a),$
 functions of the form
\[
m_{t+1,\phi}(\cdot;x,a):=\left\langle \nabla\log(p_{t+1}(\cdot|x,a)),\phi
\right\rangle +\mathrm{div}(\phi)
\]
with $\phi:$ $\mathsf{S}\rightarrow\mathbb{R}^{d}$ being a smooth and bounded
mapping with bounded derivatives. It is then not difficult to check that
\[
\int_{\mathsf{S}}p_{t+1}(y|x,a)\phi_{i}(y)\partial_{y_{i}}\log(p_{t+1}%
(y|x,a))\,dy=-\int_{\mathsf{S}}p_{t+1}(y|x,a)\partial_{y_{i}}\phi
_{i}(y)\,dy,\quad i=1,\ldots,d,
\]
and hence $m_{t+1,\phi}$ satisfies (\ref{eq:m-zeromean}) for all
$(x,a)\in\mathsf{S}\times\mathsf{A}$. This means that in
(\ref{eq:optim-dual-m}), we can take $\mathcal{M}_{t+1,x,a}^{\prime
}=\{m_{t+1,\phi}\left(  \cdot;x,a\right)  :\,\phi\in\Phi\}$ where $\Phi$ is
the linear space of mappings $\mathbb{R}^{d}\rightarrow\mathbb{R}^{d},$ which
are smooth, bounded, and with bounded derivatives. Since $\phi\rightarrow
m_{t+1,\phi}(\cdot;x,a)$ is linear in $\phi$ we moreover have that
$\mathcal{M}_{t+1,x,a}^{\prime}$ is a linear space of real valued functions.
So the problem (\ref{eq:optim-dual-m}) can be casted into a standard linear
regression problem after choosing a system of basis functions $(m_{t+1,\varphi
_{k}}\left(  \cdot;x,a\right)  )_{k\in\mathbb{N}}$ due to some basis
$(\varphi_{k})_{k\in\mathbb{N}}$ in $\Phi.$
Needles to say that the problem
(\ref{eq:optim-dual-m}) can only be solved on some finite grid, $(x_{l}%
,a_{l})_{l=1,...,L}\in\mathsf{S}\times\mathsf{A}$\ say, yielding solutions
$\phi_{k}(\cdot):=\phi(\cdot;x_{k},a_{k})$ and the corresponding martingale
functions $m_{t+1,\phi_{k}}\left(  \cdot;x_{k},a_{k}\right)  $. In order to
obtain a martingale function $m_{t+1}\equiv m_{t+1}\left(  \cdot;x,a\right)  $
for a generic pair $(x,a)$ we may apply some suitable interpolation procedure.
Loosely speaking, if $(x,a)$ is an interpolation between $(x_{k},a_{k})$ and
$(x_{k^{\prime}},a_{k^{\prime}})$ we may interpolate $\phi(\cdot;x,a)$ between
$\phi_{k}$ and $\phi_{k^{\prime}}$ correspondingly, and set $m_{t+1}%
=m_{t+1,\phi}\left(  \cdot;x,a\right)  .$ For details regarding suitable
interpolation procedures we refer to Section~\ref{convup}.
\par
Let now, for each $t\in\lbrack H[$, and $(x,a)\in\mathsf{S}\times\mathsf{A}$,
the martingale function $m_{t+1}(\cdot;x,a)$ be an approximate solution of
(\ref{eq:optim-dual-m}). Then we can construct an upper bound (upper biased
estimate) for $V_{0}^{\star}(x_{0}),$ via a standard Monte Carlo
estimate of the expectation
\begin{equation}\label{up_m}
V_{0}^{\mathrm{up}}(x)=\mathbb{E}%
_{\boldsymbol{\pi},x}\left[  \sup_{a_{\geq0}\in\mathsf{A}^{H}%
}\left(  \sum_{t=0}^{H-1}\left(  R_{t}(S_{t}(a_{\geq t}),a_{t})-m_{t+1}%
(S_{t+1}(a_{\leq t});S_{t}(a_{<t}),a_{t})\right)  +F(S_{H})\right)  \right]  .
\end{equation}
\par
Another way of constructing $\boldsymbol{\xi}\in\Xi$ is based on the assumption that the
chain $(S_{t}(a_{<t}))$  comes from  the system of the so-called \textit{random iterative
functions}:
\begin{equation}
S_{t}=\mathcal{K}_{t}(S_{t-1},a_{t-1}, \varepsilon
_{t}),\quad t\in]H], \label{eq:iterfunc-repr}%
\end{equation}
where $\mathcal{K}_{t}:$ $\mathsf{S}\times\mathsf{A}\times\mathsf{E}%
\rightarrow\mathsf{S},$ is a measurable map with $\mathsf{E}$ being a
measurable space, and $(\varepsilon_{t},t\in]H])$ is an i.i.d. sequence of
$\mathsf{E}$-valued random variables defined on a probability space
$(\Omega,\mathcal{F},\mathrm{P}).$ In this setup we may consider as the underlying
probability space  $\Omega
:=\left(  \mathsf{E}\times\mathsf{A}\right)  ^{[H]}$ instead of \eqref{omegaset}, with accordingly modified
definitions of $\mathcal{F}$ and $\left( \mathcal{F}_{t}\right) .$
\par
Let $\mathcal{P}_{\mathsf{E}}$ be the distribution of $\varepsilon_{1}$ on
$\mathsf{E},$ and assume that $(\psi_{k},\,k\in\mathbb{N}_{0})$ is a
 system in $L^{2}(\mathsf{E},\mathcal{P}_{\mathsf{E}})$ satisfying
\[
\int\psi_{k}(\varepsilon)\,d\mathcal{P}_{\mathsf{E}}=0, \quad k \in \mathbb{N}.
\]
By then letting
\begin{equation}
\eta_{t+1,K}(x,a)\equiv\eta_{t+1,K}(x,a,\varepsilon_{t+1})=\sum_{k=1}^{K}%
c_{k}(x,a)\psi_{k}(\varepsilon_{t+1})\label{etamar}%
\end{equation}
for some natural $K>0$ and ``nice'' functions
$c_k:$ $\mathsf{S}\times\mathsf{A}$ $\rightarrow$
$\mathbb{R},$ $k=1,\ldots,K,$ we have that
\[
\xi_{t+1,K}(a_{\leq t}):=\eta_{t+1,K}(S_{t}(a_{<t}),a_{t})
\]
is $\mathcal{F}_{t+1}$-measurable, and, since $\int\psi_{k}(\varepsilon
)\, d\mathcal{P}_{\mathsf{E}}(\varepsilon)=0$ for $k\in\mathbb{N}$, it holds that
$\mathbb{E}\left[  \left.  \xi_{t+1,K}(a_{\leq t})\right\vert \mathcal{F}%
_{t}\right]  =0.$ Hence, we have that $\boldsymbol{\xi}_{K}$ $=$ $(\xi
_{t+1,K}(a_{\leq t}),t\in\lbrack H[)\in\Xi.$
In this case, we can  consider the least-squares problem
\begin{equation}
\inf_{(c_{1},\ldots,c_{K})}\mathbb{E}\left[  \left(  V_{t+1}(Z^{x,a}%
)-\sum_{k=1}^{K}c_{k}\psi_{k}(\varepsilon_{t+1})\right)  ^{2}\right]  ,\quad
Z^{x,a}\equiv\mathcal{K}_{t+1}(x,a,\varepsilon_{t+1}),\label{alt}%
\end{equation}
for estimating the coefficients in (\ref{etamar}). Let us further denote
$\Sigma_{\mathsf{E},K}:=\mathbb{E}_{\varepsilon\sim\mathcal{P}_{\mathsf{E}}%
}\left[  \boldsymbol{\psi}_{K}(\varepsilon)\boldsymbol{\psi}_{K}^{\top
}(\varepsilon)\right]  $ with $\boldsymbol{\psi}_{K}(\varepsilon):=[\psi
_{1}(\varepsilon),\ldots,\psi_{K}(\varepsilon)]^\top.$ The minimization problem
(\ref{alt}) is then explicitly solved by%
\begin{equation}
\bar{\mathbf{c}}_{K}(x,a):=\Sigma_{\mathsf{E},K}
^{-1}\mathbb{E}\left[  V_{t+1}(Z^{x,a})\boldsymbol{\psi}_{K}(\varepsilon
)\right].  \label{cbalke}%
\end{equation}
In the sequel  we assume that  $\Sigma_{\mathsf{E},K}$ is known and invertible. This assumption is not particularly restrictive, as we choose the basis \(\boldsymbol{\psi}\) ourselves.   In order to compute \eqref{cbalke}, we can
construct a new sample $U_{m}(x,a)=V_{t+1}(Z_{m}^{x,a})\Sigma
_{\mathsf{E},K}  ^{-1}\boldsymbol{\psi}_{K}(\varepsilon_m)$ with
$\varepsilon_m\sim\mathcal{P}_{\mathsf{E}},$  \(Z_m^{x,a}\equiv\mathcal{K}_{t+1}(x,a,\varepsilon_m),\) $m=1,\ldots,M,$ and estimate its
mean
$\bar{\mathbf{c}}_{K}(x,a)$ by the empirical mean \begin{equation}\label{cKN}
{\mathbf{c}%
}_{K,M}(x,a)=
[c_{1,M}(x,a),\ldots,c_{K,M}(x,a)]^\top
:=
\frac{1}{M}\sum_{m=1}^{M}U_{m}(x,a).
\end{equation}
We so obtain as martingale functions in
(\ref{etamar}),
\begin{equation}\label{etamar1}
{\eta}_{t+1,K,M}:={\mathbf{c}}%
_{K,M}^{\top}(x,a)\boldsymbol{\psi}_{K}(\varepsilon_{t+1})=\sum_{k=1}^K c_{k,M}(x,a)\psi_k(\varepsilon_{t+1}).
\end{equation}
Also note that the problem (\ref{alt}) may only numerically be solved on a grid, and a suitable
interpolation procedure is required to obtain (\ref{etamar1}) for
generic $(x,a)\in\mathsf{S}\times\mathsf{A}$ (for details see Section~\ref{convup}). Finally, an
upper biased upper bound for $V_{0}^{\star}(x)$ can be obtained via an
independent standard Monte Carlo estimate of the expectation
\begin{equation}\label{up_eta}
V_{0}^{\mathrm{up}}(x)=\mathbb{E}
_{\boldsymbol{\pi},x}\left[  \sup_{a_{\geq0}\in\mathsf{A}^{H}%
}\left(  \sum_{t=0}^{H-1}\left(  R_{t}(S_{t}(a_{\geq t}),a_{t})-\eta
_{t+1,K,M}(S_{t}(a_{<t}),a_{t})\right)  +F(S_{H})\right)  \right]  .
\end{equation}
In Section~\ref{convup} we will give a detailed convergence analysis of the dual estimator (\ref{up_eta}). It is anticipated that a similar analysis can be carried out
for the dual estimator (\ref{up_m}), but this analysis is omitted due to space restrictions.

\section{Convergence analysis of the primal algorithm }
\label{sec:conv-value}
In this section, we carry out the convergence analysis of the primal algorithm designed
in Section~\ref{ValAppr}, under some mild assumptions.
\begin{assumption}\label{ass:S}
Assume that (\ref{eq:iterfunc-repr}) holds. In this case $P_{h}^{a}f(x)=\mathbb{E}_{\varepsilon\sim\mathcal{P}_{\mathsf{E}}}[f(\mathcal{K}_{h}(x,a,\varepsilon))]$,
$(x,a)\in\mathsf{S}\times\mathsf{A}.$ Also assume that the kernels \(\mathcal{K}_{h}\) are Lipschitz continuous:
\begin{equation}
|\mathcal{K}_{h}(x,a,\varepsilon)-\mathcal{K}_{h}(x^{\prime},a^{\prime
},\varepsilon)|\leq L_{\mathcal{K}}\,\rho((x,a),(x^{\prime},a^{\prime}%
)),\quad(x,a),(x^{\prime},a^{\prime})\in\mathsf{S}\times\mathsf{A}%
,\quad\varepsilon\in\mathsf{E},\label{norm}%
\end{equation}
for some constant $L_{\mathcal{K}}$ not depending on $h.$ In (\ref{norm}), the
metric $\rho\equiv\rho_{\mathsf{S}\times\mathsf{A}}$ on $\mathsf{S}%
\times\mathsf{A}$ is considered to be of the form%
\[
\rho_{\mathsf{S}\times\mathsf{A}}((x,a),(x^{\prime},a^{\prime}))=\left\Vert
\left(  \rho_{\mathsf{S}}(x,x^{\prime}),\rho_{\mathsf{A}}(a,a^{\prime
})\right)  \right\Vert ,
\]
where $\rho_{\mathsf{S}}$ and $\rho_{\mathsf{A}}$ are suitable metrics on
$\mathsf{S}$ and $\mathsf{A}$, respectively, and $\left\Vert \left(
\cdot,\cdot\right)  \right\Vert $ is a fixed but arbitrary norm on
$\mathbb{R}^{2}.$ In order to avoid an overkill of notation, we will henceforth
drop the subscripts $\mathsf{S}$, $\mathsf{A}$, and $\mathsf{S}\times
\mathsf{A}$, whenever it is clear from the arguments which metric is considered.
\end{assumption}
\begin{assumption}
\label{ass:R}
Assume that \(\sup_{(x,a)\in \mathsf{S}\times\mathsf{A}}\{|R_{h}(x,a)|\vee |F(x)|\}\leq R_{\max}\) and
\begin{eqnarray*}
\sup_{a\in \mathsf{A}}|R_{h}(x,a)-R_{h}(x',a)|\leq L_R\rho(x,x')
\end{eqnarray*}
for some constants \(R_{\max}\) and \(L_R\) not depending on \(h\in [H[.\)
\end{assumption}
We now set
\begin{equation}
\widetilde{L}_{h}:=(H-h+1)R_{\max},\text{ \ \ }h\in\lbrack H],\text{
\ \ }V_{\max}^{\star}:=\widetilde{L}_{0}=(H+1)R_{\max}.\label{Lt}%
\end{equation}

\begin{assumption}
\label{ass:gamma}
Assume that \(|\Sigma_{h,K}^{-1}\boldsymbol{\gamma}_K(x)|_{\infty}\leq \Lambda_K\)  for all \(x\in \mathsf{S}\), $h\in[H[,$ and
\begin{eqnarray*}
|\boldsymbol{\gamma}_K(x)-\boldsymbol{\gamma}_K(x')|\leq L_{\gamma,K}\rho (x,x')
\end{eqnarray*}
for a constant \(L_{\gamma,K}>0,\) where $|\cdot|$ denotes the Euclidian norm and \(|\cdot|_\infty\) stands for the \(\ell_\infty\)  norm.
\end{assumption}
Note that due to (\ref{lowconst}) and (\ref{Lt}) one has that $\left\vert V_{h,N}\right\vert
\leq\widetilde{L}_{h},$ $h\in\lbrack H],$ and that under Assumptions~\ref{ass:S}%
,~\ref{ass:R}, and \ref{ass:gamma} one has
\begin{eqnarray*}
|T_{h,N}V_{h+1,N}(x)-T_{h,N}V_{h+1,N}(x')|&\leq &  L_R\rho (x,x')+\sup_{a\in \mathsf{A}}|\widetilde P_{h+1,N}^{a}V_{h+1,N}(x)-\widetilde P_{h+1,N}^{a}V_{h+1,N}(x')|
\\
&\leq &L_R\rho (x,x')+\sup_{a\in \mathsf{A}}|\beta_{N,a}||\boldsymbol{\gamma}_K(x)-\boldsymbol{\gamma}_K(x')|
\\
&\leq & L_R\rho (x,x')+\frac{1}{N}\sum_{n=1}^{N}\sup_{a\in \mathsf{A}}|Z_{n}^{a}||\Sigma^{-1}_{h,K}\boldsymbol{\gamma}_K(X_{n})||\boldsymbol{\gamma}_K(x)-\boldsymbol{\gamma}_K(x')|
\\
&\leq & [L_R+V^{\star}_{\max}\Lambda_K\sqrt{K}  L_{\gamma,K}]\rho (x,x').
\end{eqnarray*}
Let us denote
 \(L_{V,K}:= L_R+V^{\star}_{\max}\Lambda_K  L_{\gamma,K}\sqrt{K}.\) The above estimates imply that \(V_{h,N}\in \mathrm{Lip}(L_{V,K}),\) and so the function \(f(x,a,\varepsilon) := V_{h,N}(\mathcal{K}_{h}(x,a,\varepsilon))\) satisfies
\begin{eqnarray}
\label{eq:flip}
\left|f(x,a,\varepsilon)-f(x',a',\varepsilon)\right|\leq L_{V,K}L_{\mathcal{K}}\rho((x,a),(x',a'))
\end{eqnarray}
The next assumption concerns the measures \(\mu_1,\ldots,\mu_H.\)
\begin{assumption}
\label{ass:mu}
Consider for any $h<l$ the Radon-Nikodym derivative
\begin{align*}
\mathfrak{R}_{h,l}(x^{\prime}|x,\boldsymbol{\pi}) &  :=\frac{P_{h+1}^{\pi_{h}%
}\ldots P_{l}^{\pi_{l-1}}(dx^{\prime}|x)}{\mu_l(dx^{\prime})}
\end{align*}
where for a generic policy $\boldsymbol{\pi}=(\pi_1,\ldots,\pi_H),$
\[
P_{h+1}^{\pi_{h}}(dx^{\prime}|x):=P_{h+1}(dx^{\prime}|x,\pi_{h}(x)).
\]
Assume that
\begin{equation}
\mathfrak{R}^{\max}:=\sup_{0\leq h<l<H,\boldsymbol{\pi}}\left(  \int%
\mu_h(dx)\int\mathfrak{R}_{h,l}^{2}(x^{\prime}|x,\boldsymbol{\pi})\mu_l
(dx^\prime)\right)  ^{1/2}<\infty.\label{Rinf}
\end{equation}
\end{assumption}
By the very construction of $V_{h,N}$ from $V_{h+1,N},$ $h\in\lbrack H[,$ as
outlined in Section~\ref{ValAppr}, $V_{h,N}$ may be seen as random (Lipschitz
continuous) function. In particular, for each $x\in\mathsf{S,}$ $V_{h,N}(x)$  is
measurable with respect to the $\sigma$-algebra
\begin{equation}\label{sampdata}
\mathcal{D}_{h}^{N}:=\sigma\bigl\{  \mathbf{Y}^{h;N},\ldots,\mathbf{Y}%
^{H-1;N}\bigr\}  \text{ \ \ with \ \ }\mathbf{Y}^{h;N}:=\bigl(  \bigl(
X_{1}^{h},\varepsilon_{1}^{h}\bigr)  ,\ldots,\bigl(  X_{N}^{h},\varepsilon
_{N}^{h}\bigr)  \bigr)
\end{equation}
where the pairs $\left(  X_{i}^{h},\varepsilon_{i}^{h}\right)  \sim\mu
_{h}\otimes\mathcal{P}_{\mathsf{E}}$ are i.i.d. for $h\in\lbrack H[,$
$i=1,\ldots,N,$ and Monte Carlo simulated under the measure
$\mathsf{P}\equiv\mathsf{P_N}:=\left(\mu_{h}\otimes\mathcal{P}_{\mathsf{E}}\right)^{\otimes HN}$.
\par
The following theorem provides an upper bound for the difference between \(V_{h,N}\) and \(V_{h}^{\star}.\)
\begin{theorem}
\label{thm:prim-bound}
Suppose that $\mathbb{E}_{X\sim\mu_h}\left[  |\boldsymbol{\gamma}_{K}%
(X)|^{2}\right]  \leq\varrho_{\gamma,K}^{2}$ for all $h\in[H[$. Then for $ h\in\lbrack H]$,
\begin{multline*}
 \Vert V_{h}^{\star}(\cdot)-V_{h,N}(\cdot)\Vert_{L^{2}(\mu_h\otimes\mathsf{P})}\\
\lesssim \mathfrak{R}^{\max}\left(  (H-h)\varrho_{\gamma,K}
\Lambda_K(L_{V,K}L_{\mathcal{K}}I_{\mathcal{D}}(\mathsf{A})+L_{V,K}L_{\mathcal{K}%
}\mathsf{D}(\mathsf{A})+ V_{\max}^{\star})
\sqrt{\frac
{K}{N}}+\sum_{l=h}%
^{H-1}\mathcal{R}_{K,l}\right)  ,
\end{multline*}
where
$\lesssim$ denotes $\leq$ up to an absolute constant,
$I_{\mathcal{D}}(\mathsf{A})$ is the metric entropy of $\mathsf{A}$,
$\mathsf{D}(\mathsf{A})$ is the diameter of $\mathsf{A}$ as defined in
Appendix~\ref{notation}, and%
\begin{align*}
\mathcal{R}_{K,h}  &  :=\sup_{\boldsymbol{\zeta}\in\mathbb{R}^{K\times
|\mathsf{A}|}}\mathbb{E}_{X\sim\mu_h}\left[  \sup_{a\in\mathsf{A}}\left(
\beta_{a,\boldsymbol{\zeta}}^{\top}\boldsymbol{\gamma}_{K}(X)-P_{h+1}%
^{a}V_{h+1,\boldsymbol{\zeta}}(X)\right)  ^{2}\right]  ^{1/2},
\end{align*}
where
\begin{align*}
\beta_{a,\boldsymbol{\zeta}}  &  :=\arginf_{\beta\in\mathbb{R}^{K}}%
\mathbb{E}_{X\sim\mu_h}\left[  \left(  \beta^{\top}\boldsymbol{\gamma}%
_{K}(X)-P_{h+1}^{a}V_{h+1,\boldsymbol{\zeta}}(X)\right)  ^{2}\right]
\end{align*}
with
\begin{align*}
V_{h,\boldsymbol{\zeta}}(x)  &  :=\sup_{a\in\mathsf{A}}\bigl(R_{h}%
(x,a)+\mathcal{T}_{\widetilde{L}_{h+1}}[  \zeta_{a}^{\top}\boldsymbol{\gamma
}_{K}(x)]  \bigr)\text{ \ for }0\leq h<H, \quad
V_{H,\boldsymbol{\zeta}}(x):=F(x).
\end{align*}
\end{theorem}
\paragraph{Discussion}
\begin{itemize}
\item The quantity \(\mathcal{R}_{K,h}\)  is related  to the error of approximating the conditional expectation \(P_{h+1}^{a}V_{h+1,\boldsymbol{\zeta}}\) via a linear combination of the basis functions \(\gamma_1,\ldots,\gamma_K\) in a worst case scenario, that is, for the most unfavorable  choice of  \(\boldsymbol{\zeta}.\)
Let us suppose, for illustration, that $\mathsf{A}$ is finite and take some $h<H-1.$ One then has%
\begin{align}
\mathcal{R}_{K,h}  & \leq\sum_{a\in\mathsf{A}}\sup_{\boldsymbol{\zeta}%
\in\mathbb{R}^{K\times|\mathsf{A}|}}\mathbb{E}_{X\sim\mu_{h}}\left[  \left(
\beta_{a,\boldsymbol{\zeta}}^{\top}\boldsymbol{\gamma}_{K}(X)-P_{h+1}%
^{a}V_{h+1,\boldsymbol{\zeta}}(X)\right)  ^{2}\right]  ^{1/2}
\label{conserv}
\end{align}
where $\beta_{a,\boldsymbol{\zeta}}^{\top}\boldsymbol{\gamma}_{K}$ is the
$L^{2}(\mu_{h})$ projection of $P_{h+1}^{a}V_{h+1,\boldsymbol{\zeta}}$ on
$\rm{span}(\boldsymbol{\gamma}_{K})$ with the corresponding projection error%
\begin{equation}
\mathcal{E}_{K,h}(a,\boldsymbol{\zeta}):=\mathbb{E}_{X\sim\mu_{h}}\left[\left(  \beta_{a,\boldsymbol{\zeta}}^{\top
}\boldsymbol{\gamma}_{K}(X)-P_{h+1}^{a}V_{h+1,\boldsymbol{\zeta}}(X)\right)
^{2}\right]  ^{1/2}.\label{prerror}%
\end{equation}
Under mild conditions on $P_{h+1}^{a}$, (\ref{prerror})  converges to zero
uniformly in $\boldsymbol{\zeta,}$ at a rate depending on the choice of
$\boldsymbol{\gamma}_{K}.$ For example, if the system \(\gamma_1,\gamma_2,\ldots\) is
{an orthonormal base} in $L^{2}(\mu_{h})$ then \(\max_{a\in \Aset}\sup_{\boldsymbol{\zeta}%
\in\mathbb{R}^{K\times|\mathsf{A}|}}\mathcal{E}_{K,h}(a,\boldsymbol{\zeta})\lesssim K^{-\beta}\),
{$\beta>0$,} provided that the series
\begin{eqnarray*}
 \sum_{k=1}^\infty k^{2\beta} \,\mathbb{E}_{X\sim \mu_h} [\gamma_{k}(X)P_{h+1}^{a}V_{h+1,\boldsymbol{\zeta}}(X)]^2
\end{eqnarray*}
is uniformly bounded in \(\boldsymbol{\zeta}%
\in\mathbb{R}^{K\times|\mathsf{A}|}\) and \(a\in \Aset.\)
Hence, then $\mathcal{R}_{K,h}$ $\lesssim
|\Aset|K^{-\beta}$ $\rightarrow0$ for
$K\rightarrow\infty.$ Note that  (\ref{conserv}%
) {is a worst case estimate}, which may be very rough in general.

\item  Suppose that \(P_{h+1}^{a}=:P_{h+1}\) does not
depend on \(a\in \Aset\), and that
  \(\gamma_1,\gamma_2,\ldots \) are bounded eigenfunctions (corresponding to nonnegative eigenvalues) of \(P_{h+1}\). Let further     \(F(x)=\beta^\top \boldsymbol{\gamma}_K(x)\) for some \(\beta\in \mathbb{R}^K\) and \(R_{t}(x,a)=R_{1,t}(x) R_{2,t}(a)\)
with \(R_{1,t}(x)=c^\top_t \boldsymbol{\gamma}_K(x)\geq 0,\)  then for \(\widetilde{L}_{h+1}\)  large enough,
\(\mathcal{R}_{K,h}=0\) (in this case we may take \(\zeta_{a}\) independent of \(a\) in the definition of \(V_{h+1,\boldsymbol{\zeta}}\)) and only the stochastic part of the error remains:
\begin{eqnarray}
\label{eq:zerobias}
\Vert V_{h}^{\star}-V_{h,N}\Vert_{L^{2}(\mu_h\otimes\mathsf{P})}
\lesssim  H\mathfrak{R}^{\max} \varrho_{\gamma,K}\Lambda_K
(L_{V,K}L_{\mathcal{K}}I_{\mathcal{D}}(\mathsf{A})+L_{V,K}L_{\mathcal{K}%
}\mathsf{D}(\mathsf{A})+ V_{\max}^{\star})
\sqrt{\frac{K}{N}}.
\end{eqnarray}
\item Let us consider the stochastic error \eqref{eq:zerobias} in more detail
for an example where \(\Aset=[0,1]^{\da}\) for some \(\da\in \mathbb{N}.\) One then has \(\mathsf{D}(\mathsf{A})=\sqrt{\da}\) and \(I_{\mathcal{D}}(\mathsf{A})\lesssim \sqrt{\da}.\) In this example the bound \eqref{eq:zerobias} depends sub-linearly in \(\da.\) If in addition all basis functions \((\gamma_k)\)  are uniformly bounded and the infinity matrix norm (i.e. the maximum absolute row sum) of \(\Sigma_{h,K}\)
is uniformly bounded from below
for all \(K\in \mathbb{N}\) and \(h\in [H],\)
then \(\varrho_{\gamma,K}  \lesssim K^{1/2},\) \(\Lambda_K\lesssim 1,\) \(L_{V,K}\lesssim L_{\gamma,K}H K^{1/2}\), $V_{\max}^{\star}\lesssim H$, and 
the bound in Theorem~\ref{thm:prim-bound}
transforms to
\begin{eqnarray}
\label{eq:zerobias1}
\Vert V_{h}^{\star}-V_{h,N}\Vert_{L^{2}(\mu_h\otimes\mathsf{P})}
\lesssim
\frac{(H-h)H\mathfrak{R}^{\max}\sqrt{\da}
L_{\gamma,K} K^{3/2}}{\sqrt{N}}
+\mathfrak{R}^{\max}\sum_{l=h}%
^{H-1}\mathcal{R}_{K,l},
\end{eqnarray}
where \(\lesssim\) means inequality up to a constant not depending on \(H,\) \(N,\) \(K\) and \(\Aset.\)
{Another relevant situation is the case of finite  \(\Aset.\) Here \( I_{\mathcal{D}}(\mathsf{A})=\sqrt
{\log|\mathsf{A}|}\) and \(\mathsf{D}(\mathsf{A})=1.\) Hence \eqref{eq:zerobias1} changes to
 \begin{eqnarray}
\label{eq:zerobias2}
\Vert V_{h}^{\star}-V_{h,N}\Vert_{L^{2}(\mu_h\otimes\mathsf{P})}
\lesssim
\frac{(H-h)H\mathfrak{R}^{\max}\sqrt
{\log|\mathsf{A}|}
L_{\gamma,K} K^{3/2}}{\sqrt{N}}
+\mathfrak{R}^{\max}\sum_{l=h}%
^{H-1}\mathcal{R}_{K,l}.
\end{eqnarray}
Let us point out to a logarithmic dependence of \eqref{eq:zerobias2} on \(|\mathsf{A}|.\)
}
\item Let us remark on Assumption~\ref{ass:mu} and discuss the quantity \(\mathfrak{R}^{\max}\). Consider $\mathsf{S}=\mathbb{R}^{d}$ and assume that the
transition kernels are absolutely continuous with respect  to the Lebesgue measure on \(\mathbb{R}^{d}\), that is,
\begin{gather*}
P_{h+1}^{\pi_{h}}\cdot\ldots \cdot P_{l}^{\pi_{l-1}}(dy|x)=p_{h+1}^{\pi_{h}}\cdot\ldots
\cdot p_{l}^{\pi_{l-1}}(y|x)\,dy.
\end{gather*}
Further assume that
\[
\sup_{0\leq h<l<H,\boldsymbol{\pi}}p_{h+1}^{\pi_{h}}\ldots p_{l}^{\pi_{l-1}%
}(y|x)\leq Ce^{-\alpha_H\left\vert y-x\right\vert ^{2}}\text{ \ for some }C,\alpha_H>0,
\]
and consider absolutely continuous reference measures $\mu_{h}(dx)=\mu
_{h}(x)\,dx.$ For the bound (\ref{Rinf}), we then have%
\begin{gather*}
\left(  \mathfrak{R}^{\max}\right)  ^{2}=\sup_{0\leq h<l<H,\boldsymbol{\pi}%
}\int\int\frac{\mu_{h}(x)}{\mu_{l}(y)}\left(  P_{h+1}^{\pi_{h}}\ldots
P_{l}^{\pi_{l-1}}(y|x)\right)  ^{2}\,dx \, dy\\
\leq C^{2}\max_{0\leq h<l<H}\int\int\frac{\mu_{h}(x)}{\mu_{l}(x+u)}%
e^{-2\alpha_H\left\vert u\right\vert ^{2}}\, dx\, du.
\end{gather*}
The latter expression can be easily bounded by choosing $\mu_{h}$ to be Gaussian
with an appropriate variance structure depending on $h.$ For example,
take $d=1$ and consider
\[
\mu_{h}(x)=\sqrt{\frac{{\alpha_H}}{{\pi (h+1)}}}e^{-\frac{\alpha_H}{h+1}x^{2}},\text{ \ \ }%
h\in[H[,
\]
then straightforward calculations yield
$$\mathfrak{R}^{\max}\leq
C\sqrt{\max_{0\leq h<l<H}\frac{(l+1)\pi}{\alpha_H\sqrt{2(l-h)-1}}}\leq C\sqrt{\frac
{H\pi}{\alpha_H}}.%
$$
If \(\alpha_H\) is polynomial in \(H,\) then   the bound of Theorem~\ref{thm:prim-bound} also grows polynomially in  \(H\) as opposed to the most bounds available in the literature. Also note that this bound is obtained under rather general assumptions on the sets \(\Xset\) and \(\Aset.\) In particular, we don't assume that either \(\Xset\) or \(\Aset\) is finite.
\end{itemize}

\section{Convergence analysis of the dual algorithm}
\label{convup}

\subsection{Convergence of martingale functions}

For the dual representation (\ref{up_eta}) we construct an $H$-tuple of
martingale functions $\boldsymbol{\widetilde{\eta}}:=(\widetilde{\eta}_{t+1,K,M}%
(x,a),\,t\in\lbrack H[),$ see  (\ref{etamar1})
as outlined in
Section~\ref{dualalg}, from a given pre-computed
$H$-tuple of approximate value functions $(V_{t+1,N},\,t\in\lbrack H[)$ based on sampled data  $\mathcal{D}^{N}_1$ (see (\ref{sampdata})), and a system of  \(K_\text{pr}\) basis functions $\boldsymbol{\gamma}_{K_\text{pr}}$, as outlined in
Section~\ref{ValAppr}.

Let us consider a fixed time $t\in\lbrack H[$ and suppress time subscripts where
notationally convenient. We fix two (random) grids $\mathsf{S}_{L}:=\{x_{1}%
,\ldots,x_{L}\}$ and $\mathsf{A}_{L}:=\{a_{1},\ldots,a_{L}\}$ on $\mathsf{S}$
and $\mathsf{A},$ respectively, and obtain  values of the coefficient functions \(c_{k,M}\) on \(\mathsf{S}_{L}\times \mathsf{A}_{L}\) due to $M$
simulations. Next, we construct
\[
\eta_{t+1,K,M}(x,a)\equiv\eta_{t+1,K,M}(x,a,\varepsilon)
=\mathbf{c}_{K,M}^{\top}%
(x,a)\boldsymbol{\psi}(\varepsilon)
=:\sum_{k=1}^{K}%
c_{k,M}(x,a)\psi_{k}(\varepsilon),
\]
for $(x,a)\in\mathsf{S}_{L}\times\mathsf{A}_{L}.$
To approximate $\eta_{t+1,K,M}(x,a)$ for $(x,a)\not \in \mathsf{S}_{L}%
\times\mathsf{A}_{L},$ we suggest to use an appropriate interpolation
procedure described below, which is particularly useful for our situation
where the function to be interpolated is only Lipschitz continuous (due to the
presence of the maximum). The \emph{optimal central interpolant} for a
function $f\in\mathrm{Lip}_{\rho}(\mathcal{L})$ on \(\mathsf{S}\times\mathsf{A}\) with respect to some metric
$\rho$ on $\mathsf{S}\times\mathsf{A}$ is defined as
\[
I[f](x,a):=(H_{f}^{\mathrm{low}}(x,a)+H_{f}^{\mathrm{up}}(x,a))/2,
\]
where
\begin{align*}
H_{f}^{\mathrm{low}}(x,a) &  :=\max_{(x^{\prime},a^{\prime})\in\mathsf{S}%
_{L}\times\mathsf{A}_{L}}(f(x^{\prime},a^{\prime})-\mathcal{L}\rho
((x,a),(x^{\prime},a^{\prime})),\\
H_{f}^{\mathrm{up}}(x,a) &  :=\min_{(x^{\prime},a^{\prime})\in\mathsf{S}%
_{L}\times\mathsf{A}_{L}}(f(x^{\prime},a^{\prime})+\mathcal{L}\rho
((x,a),(x^{\prime},a^{\prime})).
\end{align*}
Note that $H_{f}^{\mathrm{low}}(x,a)\leq f(x,a)\leq H_{f}^{\mathrm{up}}(x,a),$
$H_{f}^{\mathrm{low}},H_{f}^{\mathrm{up}}\in\mathrm{Lip}_{\rho}(\mathcal{L})$
and hence $I[f]\in\mathrm{Lip}_{\rho}(\mathcal{L}).$ An efficient algorithm to
compute the values of the interpolant $I[f]$ without knowing $\mathcal{L}$ in
advance can be found in \cite{beliakov2006interpolation}. The so constructed
interpolant achieves the bound
\begin{align}
\Vert f-I[f]\Vert_{\infty} &  \leq\mathcal{L}\rho_{L}(\mathsf{S}%
,\mathsf{A})\label{eq:interp-error}\\
&  :=\mathcal{L}\sup_{(x,a)\in\mathsf{S}\times\mathsf{A}}\min_{(x^{\prime
},a^{\prime})\in\mathsf{S}_{L}\times\mathsf{A}_{L}}\rho((x,a),(x^{\prime
},a^{\prime})).\nonumber
\end{align}
The quantity $\rho_{L}(\mathsf{S},\mathsf{A})$ is usually called covering
radius (also known as the mesh norm or fill radius) of $\mathsf{S}_{L}%
\times\mathsf{A}_{L}$ with respect to $\mathsf{S}\times\mathsf{A}$. We
set
\begin{equation}\label{etatil}
\widetilde{\eta}_{t+1,K,M}(x,a)\equiv
\widetilde{\eta}_{t+1,K,M}(x,a,\varepsilon):=\sum_{k=1}^{K}\widetilde{c}_{k,M}(x,a)\psi
_{k}(\varepsilon)\text{ \ \ with}\quad\widetilde{c}_{k,M}:=I[c_{k,M}].
\end{equation}
The coefficients $\widetilde{c}_{k,M}(x,a)$
in (\ref{etatil}) are considered as random, which
are measurable with respect to
$\mathcal{D}_{t+1}^{N}\vee \mathcal{G}_{t+1}^{M}$ with $\mathcal{G}%
_{t+1}^{M}:=\sigma \left\{ \widetilde{\varepsilon }_{1}^{t+1},\ldots ,%
\widetilde{\varepsilon }_{M}^{t+1}\right\} ,$ where $\widetilde{\varepsilon }%
_{m}^{t+1}\sim \mathcal{P}_{\mathsf{E}},$ $m=1,\ldots ,M$, $t\in [H[$, denote the i.i.d.
random  drawings used in  (\ref{cKN}). Let us denote the simulation measure (for both primal and dual) with $\mathsf{P}\equiv\mathsf{P}_{N,M}:=
\mathsf{P}_{N}\otimes \mathcal{P}_{\mathsf{E}}^{\otimes HM}$
(while slightly abusing notation) with \(\mathsf{P_N}=\left(\mu_{h}\otimes\mathcal{P}_{\mathsf{E}}\right)^{\otimes HN}\).

Furthermore, denote by ${\mathbf{c}}_{K}(x,a)$ $=$ $[{c}_{1}(x,a),\ldots,{c}_{K}(x,a)]^\top$ the unique solution of the
minimization problem
\begin{equation}\label{cK}
\inf_{c_{1},\ldots,c_{K}}\mathbb{E}_{\varepsilon\sim\mathcal{P}_{\mathsf{E}}%
}\left[\left(  V_{t+1}^{\star}(\mathcal{K}%
_{t+1}(x,a,\varepsilon))-\sum_{k=1}^{K}c_{k}\psi_{k}(\varepsilon)\right)  ^{2}\right]%
\end{equation}
for any $(x,a)\in\mathsf{S}\times\mathsf{A,}$ and define \(
{\eta}_{t+1,K}(x,a):=\mathbf{c}^\top_K(x,a)\boldsymbol{\psi
}_{K}(\varepsilon).\)
As such, ${\eta}_{t+1,K}(x,a)$ is the projection  of the optimal martingale function ${\eta}^{\star}_{t+1}(x,a)$
on \(\mathrm{span}(\psi_1,\ldots,\psi_K)\).
\begin{assumption}
\label{ass:gamma1} Assume that $|\Sigma_{\mathsf{E},K}^{-1}\boldsymbol{\psi
}_{K}(\varepsilon)|_\infty\leq\Lambda_{\mathsf{E},K}$ for all $\varepsilon
\in\mathsf{E}$, and that  $\mathbb{E}_{\varepsilon\sim\mathcal{P}_{\mathsf{E}}%
}\left[  |\boldsymbol{\psi}_{K}(\varepsilon)|^{2}\right]  \leq\varrho_{\psi
,K}^{2}$.
\end{assumption}
\par
The following theorem provides a bound on the difference between the  projection
${\eta}_{t+1,K}(x,a)$ and its estimate (\ref{etatil}).
\begin{theorem}\label{dualuperr}
Under Assumptions~\ref{ass:S},~\ref{ass:R},~ \ref{ass:gamma}, and \ref{ass:gamma1} it holds
that
\begin{multline*}
\mathbb{E}_{\mathcal{P}_{\mathsf{E}}\otimes\mathsf{P}}\left[  \sup_{(x,a)\in\mathsf{S}\times\mathsf{A}}\left\vert
{\eta}_{t+1,K}(x,a,\cdot)-\widetilde{\eta}_{t+1,K,M}(x,a,\cdot)\right\vert ^{2}\right]
\lesssim\\
\varrho_{\psi,K}^{2}\frac{K(L_{V,K_\text{pr}}L_{\mathcal{K}}I_{\mathcal{D}%
}(\mathsf{S}\times\mathsf{A})+L_{V,K_\text{pr}}L_{\mathcal{K}}\mathsf{D}(\mathsf{S}%
\times\mathsf{A})+V_{\max}^{\star})^2\Lambda_{\mathsf{E},K}^2 }{M}\\
+K\Lambda_{\mathsf{E},K}^{2}\varrho_{\psi,K}^{2}\sup_{(x,a)\in\mathsf{S}\times\mathsf{A}%
}\left\Vert \frac{dP_{t+1}(\cdot|x,a)}{d\mu_{t+1}(\cdot)}\right\Vert
_{\mathsf{\infty}}\Vert V_{t+1}^{\star}-V_{t+1,N}\Vert_{L^{2}(\mu_{t+1}\otimes\mathsf{P})}^{2}\\
+K\varrho_{\psi,K}^{2}L_{V,K_\text{pr}}^{2}L_{\mathcal{K}}^{2}\Lambda_{\mathsf{E},K}^{2}\rho_{L}%
^{2}(\mathsf{S},\mathsf{A}),
\end{multline*}
where
$\lesssim$ denotes $\leq$ up to a natural constant,
the constants $L_{V,K_\text{pr}},$ $L_{\mathcal{K}},$ and the measure $\mu_{t+1}$ are inferred from the primal procedure in Section~\ref{sec:conv-value}.
\end{theorem}
Let us now consider the approximation error
\begin{align*}
\mathcal{E}_{K,t}^{2}  & :=\mathbb{E}_{\varepsilon\sim\mathcal{L}%
_{\mathsf{E}}}\left[  \sup_{(x,a)\in\mathsf{S}\times\mathsf{A}}\left\vert
{\eta}_{t+1,K}(x,a)-\eta_{t+1}^{\star}(x,a)\right\vert^{2} \right]  \end{align*}
 with
 \begin{align*}
\eta_{t+1}^{\star}(x,a)  & = V_{t+1}^{\star}(\mathcal{K}%
_{t+1}(x,a,\varepsilon))-\mathbb{E}\left[  V_{t+1}^{\star}(\mathcal{K}%
_{t+1}(x,a,\varepsilon))\right]  ,\quad(x,a)\in\mathsf{S}\times\mathsf{A}%
,\quad t\in\lbrack H[.
\end{align*}
Suppose that one has pointwise
\[
\eta_{t+1}^{\star}(x,a)=\sum_{k=1}^{\infty}c_{k,t+1}^{\star}(x,a)\psi
_{k}(\varepsilon_{t+1}),\quad(x,a)\in\mathsf{S}\times\mathsf{A},\quad
t\in\lbrack H[.
\]
 If
$\Vert\psi_{k}\Vert_{\infty}\leq\psi_{k}^{\ast}$ for all $k\in\mathbb{N},$
then
\begin{multline*}
\mathcal{E}_{K,t}^{2}=\mathbb{E}\left[  \sup_{(x,a)\in\mathsf{S}%
\times\mathsf{A}}\left\vert \sum_{k=K+1}^{\infty}c_{k,t+1}^{\star}%
(x,a)\psi_{k}(\varepsilon_{t})\right\vert ^{2}\right]
\leq\sup_{(x,a)\in\mathsf{S}\times\mathsf{A}}\left(  \sum_{k=K+1}^{\infty
}|c_{k,t+1}^{\star}(x,a)|\psi_{k}^{\ast}\right)^2  .
\end{multline*}
If
\begin{equation}\label{eq:bias-eta}
\sup_{(x,a)\in\Xset\times\Aset}\sum_{k=1}^{\infty}k^{\beta_{\psi}}|c_{k,t+1}^{\star
}(x,a)|\psi_{k}^{\ast}\leq C<\infty
\end{equation}
for some $\beta_{\psi}>0,$ then
\begin{equation}
\mathcal{E}_{K,t}^{2}\leq C^{2}K^{-2\beta_{\psi}}.\label{appr_fehler}%
\end{equation}
\paragraph{Discussion}
\begin{itemize}
\item Let us discuss the quantity \(\rho_{L}(\mathsf{S},\mathsf{A}).\) Let $\Xset = [0,1]^\dx,$ $\Aset = [0,1]^\da$  for some  \(\dx,\da\in \mathbb{N}\) and let the points \(\Xset_L\) (\(\Aset_L\)) be uniformly distributed on \(\Xset\) (\(\Aset\)).
Moreover set, \(\rho((x,a),(x',a'))=|x-x'|+|a-a'|.\) Then, similarly to \cite{Reznikov} it can be shown that
\begin{equation}
\label{cov rad lp}
[\mathbb{E}{\rho^p_L(\Xset\times \Aset)}]^{1/p} \lesssim \sqrt{\dx} \left( \frac{p \log L}{L} \right)^{1/\dx}+\sqrt{\da} \left( \frac{p \log L}{L} \right)^{1/\da},
\end{equation}
where \(\lesssim\) stands for inequality up to a constant not depending on \(L.\) Using the Markov inequality, we can derive a high probability bound  for \(\rho_{L}(\mathsf{S},\mathsf{A}).\)   Note that if \(\Xset\) and \(\Aset\) are finite we need not to interpolate and \(\rho_L=0\).

\item Assume that  all basis functions \((\psi_k)\)  are uniformly bounded and that the
infinity matrix norm (i.e. the maximum absolute row sum) of
the matrix \(\Sigma_{\mathsf{E},K}\) is uniformly bounded from below for all \(K\in \mathbb{N}.\) In this case, \(\varrho_{\psi,K} \lesssim K^{1/2},\) \(\Lambda_{\mathsf{E},K}\lesssim 1,\) \(L_{V,K_\text{pr}}\lesssim L_{\gamma,K_\text{pr}}H K_\text{pr}^{1/2}.\) Suppose also that   the quantities \(\sup_{(x,a)\in\mathsf{S}\times\mathsf{A}%
}\left\Vert \frac{dP_{t+1}(\cdot|x,a)}{d\mu_{t+1}(\cdot)}\right\Vert
_{\mathsf{\infty}}\) are uniformly bounded for all \(t>0.\) Then using the bound \eqref{eq:zerobias1} and the bound of Theorem~\ref{dualuperr}, we arrive at
\begin{multline}
\mathbb{E}_{\mathcal{P}_{\mathsf{E}}\otimes\mathsf{P}}
\left[  \sup_{(x,a)\in\mathsf{S}\times\mathsf{A}}\left\vert {\eta
}_{t+1}^{\star}(x,a,\cdot)-\widetilde{\eta}_{t+1,K,M}(x,a,\cdot)\right\vert ^{2}\right]
^{1/2}\lesssim D_{t+1}(H,K,K_{\text{pr}},L)+\label{eq:zerobias-up}\\
\frac{HKK_{\text{pr}}^{1/2}L_{\gamma,K_{\text{pr}}}(\sqrt{\da}+\sqrt{\dx}%
)}{\sqrt{M}}+\frac{(H-t-1)HKK_{\text{pr}}^{3/2}\mathfrak{R}^{\max}%
L_{\gamma,K_{\text{pr}}}\sqrt{\da}}{\sqrt{N}}
\end{multline}
where $D_{t+1}(H,K,K_{\text{pr}},L)$ denotes the deterministic part of the
error reflecting the approximation properties of the systems
$\boldsymbol{\gamma}_{K_{\text{pr}}},$ $\boldsymbol{\psi}_{K}$ and the interpolation error
due to finite $L$ (see the above discussions for some quantitative estimates).
Under the above assumptions, including (\ref{eq:bias-eta}), one obtains from (\ref{appr_fehler}),
Theorem~\ref{thm:prim-bound}, and
Theorem~\ref{dualuperr},%
\[
D_{t+1}(H,K,K_{\text{pr}},L)\lesssim K^{-\beta_{\psi}}+HKK^{1/2}_{\rm{pr}}L_{\gamma
,K_{\text{pr}}}\rho_{L}(\mathsf{S},\mathsf{A})+K\mathfrak{R}^{\max}%
\sum_{l=t+1}^{H-1}\mathcal{R}_{K_{\text{pr}},l},
\]
where $\lesssim$ means inequality up to a constant not depending on $H,$ $N,$
$K,$ $K_{\text{pr}},$ and $L.$ This bound is again polynomial in $H,$ provided
that $\mathfrak{R}^{\max}$ depends polynomially on $H$ (see the discussion
after Theorem~\ref{thm:prim-bound}).

\end{itemize}

\subsection{Convergence of upper bounds}
Suppose that the estimates $\boldsymbol{\widetilde{\eta}}=(\widetilde{\eta
}_{t+1}(x,a),\,t\in\lbrack H[)$ of the optimal martingale tuple $\eta^{\star
}=(\eta_{t}^{\star}(x,a),\,t\in]H])$ are constructed based on the sampled data
$\mathcal{D}_{1}^{N}\vee \mathcal{G}_{1}^{M}\vee\ldots\vee \mathcal{G}_{H}^{M}$ such that Theorem~\ref{dualuperr} holds. Consider for
$\widetilde{\boldsymbol{\xi}}:=(\widetilde{\eta}_{t+1}(S_{t}(a_{<t}%
),a_{t}),\, a_{<t}\in \mathsf{A}^t,\, t\in\lbrack H[)\in\Xi$, $S_0=x$, the upper bias
\begin{align*}
V_{0}^{\mathrm{up}}(x;\widetilde{\boldsymbol{\xi}})-V_{0}^{\star}(x) &
=\mathbb{E}_{x}\left[  \sup_{a_{\geq0}\in\mathsf{A}^{H}}\left(  \sum
_{t=0}^{H-1}\left(  R_{t}(S_{t}(a_{<t}),a_{t})-\widetilde{\eta}_{t+1}%
(S_{t}(a_{<t}),a_{t})\right)  +F(S_{H})\right)  \right]  \\
&  -\mathbb{E}_{x}\left[  \sup_{a_{\geq0}\in\mathsf{A}^{H}}\left(  \sum
_{t=0}^{H-1}\left(  R_{t}(S_{t}(a_{<t}),a_{t})-\eta_{t+1}^{\star}(S_{t}%
(a_{<t}),a_{t})\right)  +F(S_{H})\right)  \right]  \\
&  \leq\mathbb{E}_{x}\left[  \sup_{a_{\geq0}\in\mathsf{A}^{H}}\left\vert
\sum_{t=0}^{H-1}\eta_{t+1}^{\star}(S_{t}(a_{<t}),a_{t})-\sum_{t=0}%
^{H-1}\widetilde{\eta}_{t+1}(S_{t}(a_{<t}),a_{t})\right\vert \right]  \\
&  \leq\sum_{t=0}^{H-1}\mathbb{E}_{x}\left[  \sup_{(x,a)\in\mathsf{S}%
\times\mathsf{A}}\left\vert \eta_{t+1}^{\star}(x,a)-\widetilde{\eta}%
_{t+1}(x,a)\right\vert \right]  \\
&  \leq\sum_{t=0}^{H-1}\mathbb{E}_{x}\left[  \sup_{(x,a)\in\mathsf{S}%
\times\mathsf{A}}\left\vert \eta_{t+1}^{\star}(x,a)-\widetilde{\eta}%
_{t+1}(x,a)\right\vert ^{2}\right]  ^{1/2},
\end{align*}
where $\mathbb{E}_{x}$ denotes the ``all-in'' expectation, i.e. including
the randomness of the pre-simulation, and, the independently simulated trajectories
$t\rightarrow S_t(a_{<t})$. Furthermore, similarly,
\begin{multline*}
\mathrm{Var}\left[  \sup_{a_{\geq0}\in\mathsf{A}^{H}}\left(  \sum_{t=0}%
^{H-1}\left(  R_{t}(S_{t}(a_{<t}),a_{t})-\widetilde{\eta}_{t+1}(S_{t}%
(a_{<t}),a_{t})\right)  +F(S_{H})\right)  \right]  \\
=\mathrm{Var}\left[
\begin{array}
[c]{c}%
\sup_{a_{\geq0}\in\mathsf{A}^{H}}\left(  \sum_{t=0}^{H-1}\left(  R_{t}%
(S_{t}(a_{<t}),a_{t})-\widetilde{\eta}_{t+1}(S_{t}(a_{<t}),a_{t})\right)
+F(S_{H})\right)  \\
-\sup_{a_{\geq0}\in\mathsf{A}^{H}}\left(  \sum_{t=0}^{H-1}\left(  R_{t}%
(S_{t}(a_{<t}),a_{t})-\eta_{t+1}^{\star}(S_{t}(a_{<t}),a_{t})\right)
+F(S_{H})\right)
\end{array}
\right]  \\
\leq\mathbb{E}_x\left[  \left(  \sum_{t=0}^{H-1}\sup_{(x,a)\in\mathsf{S}%
\times\mathsf{A}}\left\vert \eta_{t+1}^{\star}(x,a)-\widetilde{\eta}%
_{t+1}(x,a)\right\vert \right)  ^{2}\right]  .
\end{multline*}
Hence for the standard deviation we get by the triangle inequality,%
\begin{multline*}
\mathrm{Dev}\left[  \sup_{a_{\geq0}\in\mathsf{A}^{H}}\left(  \sum_{t=0}%
^{H-1}\left(  R_{t}(S_{t}(a_{<t}),a_{t})-\widetilde{\eta}_{t+1}(S_{t}%
(a_{<t}),a_{t})\right)  +F(S_{H})\right)  \right]  \\
\leq\sum_{t=0}^{H-1}\mathbb{E}_x\left[  \sup_{(x,a)\in\mathsf{S}\times
\mathsf{A}}\left\vert \eta_{t+1}^{\star}(x,a)-\widetilde{\eta}_{t+1}%
(x,a)\right\vert ^{2}\right]  ^{1/2}.
\end{multline*}
Thus, for the Monte Carlo estimate of $V_{0}^{\mathrm{up}}(x;\widetilde
{\boldsymbol{\xi}})$,%

\[
V_{0,N_{\mathrm{{test}}}}^{\mathrm{up}}(x;\widetilde{\boldsymbol{\xi}}%
)=\frac{1}{N_{\mathrm{{test}}}}\sum_{n=1}^{N_{\mathrm{{test}}}}\sup_{a_{\geq
0}\in\mathsf{A}^{H}}\left(  \sum_{t=0}^{H-1}\left(  R_{t}(S_{t}^{(n)}%
(a_{<t}),a_{t})-\widetilde{\eta}_{t+1}(S_{t}^{(n)}(a_{<t}),a_{t})\right)
+F(S_{H}^{(n)})\right)
\]
with
\[
S_{t}^{(n)}(a_{<t})=\mathcal{K}_{t}(S_{t-1}^{(n)}(a_{<t-1}),a_{t-1}%
,\varepsilon_{t}^{(n)}),\quad t\in]H],\quad S_{0}^{(n)}=x,
\]
we obtain
\begin{align*}
\mathbb{E}_x\left[  |V_{0,N_{\mathrm{{test}}}}^{\mathrm{up}}(x;\widetilde
{\boldsymbol{\xi}})-V_{0}^{\star}(x)|^{2}\right]  ^{1/2} &  \leq
\mathbb{E}_x\left[  |V_{0,N_{\mathrm{{test}}}}^{\mathrm{up}}(x;\widetilde
{\boldsymbol{\xi}})-V_{0}^{\mathrm{up}}(x;\widetilde{\boldsymbol{\xi}}%
)|^{2}\right]  ^{1/2}+V_{0}^{\mathrm{up}}(x;\widetilde{\boldsymbol{\xi}%
})-V_{0}^{\star}(x)\\
&  \leq\frac{1}{\sqrt{N_{\mathrm{{test}}}}}\sum_{t=0}^{H-1}\mathbb{E}_x\left[
\sup_{(x,a)\in\mathsf{S}\times\mathsf{A}}\left\vert \eta_{t+1}^{\star
}(x,a)-\widetilde{\eta}_{t+1}(x,a)\right\vert ^{2}\right]  ^{1/2}\\
&  +\sum_{t=0}^{H-1}\mathbb{E}_{x}\left[  \sup_{(x,a)\in\mathsf{S}%
\times\mathsf{A}}\left\vert \eta_{t+1}^{\star}(x,a)-\widetilde{\eta}%
_{t+1}(x,a)\right\vert \right]  \\
&  \leq \left(\frac{1}{\sqrt{N_{\mathrm{{test}}}}}+1\right)\sum_{t=0}^{H-1}\mathbb{E}%
_{x}\left[  \sup_{(x,a)\in\mathsf{S}\times\mathsf{A}}\left\vert \eta
_{t+1}^{\star}(x,a)-\widetilde{\eta}_{t+1}(x,a)\right\vert ^{2}\right]
^{1/2}.
\end{align*}
By using the bound of Theorem~\ref{dualuperr} (or in more specific form \eqref{eq:zerobias-up}), we derive the corresponding bound \(\mathbb{E}_x\left[  |V_{0,N_{\mathrm{{test}}}}^{\mathrm{up}}(x;\widetilde
{\boldsymbol{\xi}})-V_{0}^{\star}(x)|^{2}\right]  ^{1/2}.\) Note that this bound remains polynomial in \(H\) under rather general assumptions.  Let us also remark that we can use the same interpolation points to construct \(\widetilde{\eta}_{t+1}(S_{t}^{(n)}(a_{<t}),a_{t})\) for all \(n=1,\ldots,N\) and all \(a_{\leq t}.\)


\ifdefined\usejmlrstyle

\acks{
  We thank
  }

\else

\subsection*{Acknowledgments}
J.S. gratefully acknowledges financial support from the German science foundation (DFG) via the
cluster of excellence MATH+, project AA4-2.
\fi

\section{Proofs}

\subsection{Proof of Theorem~\ref{thm:prim-bound}}

\textit{One-step analysis:} Suppose that after $h$ steps of the algorithm the
estimates $V_{H,N},\ldots,V_{h+1,N}$ of the value functions $V_{H}^{\star
},\ldots,V_{h+1}^{\star}$, respectively, are constructed using sampled data
$\mathcal{D}_{h+1}^{N},$ such that $\Vert V_{t,N}\Vert_{\infty}\leq 
\widetilde{L}_t
\leq
V_{\max
}^{\star}$ a.s. for all $t=h+1,\ldots,H.$ Denote for $a\in\mathsf{A}$,
\begin{align*}
\ell^{a}(\beta)  & :=\mathbb{E}\left[  (Z^{a}-\beta^{\top}\boldsymbol{\gamma
}_{K}(X))^{2}\mid{\mathcal{D}_{h+1}^{N}}\right]  \text{ \ \ with \ }\\
\text{\ }Z^{a}  & \sim V_{h+1,N}(Y^{a,X}),\quad Y^{a,X}\sim P_{h+1}%
(\cdot|X,a),\text{ \ \ }X\sim\mu_{h}.
\end{align*}
The unique minimizer of $\ell^{a}(\beta)$ is given by the ${\mathcal{D}%
_{h+1}^{N}}$-measurable vector
\[
\beta_{a}=\mathbb{E}\left[  Z^{a}\Sigma^{-1}\boldsymbol{\gamma}_{K}%
(X)\mid{\mathcal{D}_{h+1}^{N}}\right]  =\mathbb{E}_{X\sim\mu_{h}}\left[
P_{h+1}^{a}V_{h+1,N}(X)\Sigma^{-1}\boldsymbol{\gamma}_{K}(X)\mid
{\mathcal{D}_{h+1}^{N}}\right]  .
\]
For the estimation of the ${\mathcal{D}_{h}^{N}}$-measurable vector
$\beta_{N,a}$ in (\ref{7++}), see (\ref{lowconst}), (\ref{7+}), and
Assumption~\ref{ass:S}, it then holds that
\begin{multline*}
\mathbb{E}_{\mu_{h}\otimes\mathsf{P}}\left[  \sup_{a\in\mathsf{A}}\left(
(\beta_{N,a}^{\top}-\beta_{a}^{\top})\boldsymbol{\gamma}_{K}(X)\right)
^{2}\mid{\mathcal{D}_{h+1}^{N}}\right]  \\
\leq\mathbb{E}_{\mathsf{P}}\left[  \sup_{a\in\mathsf{A}}\left\vert \beta
_{N,a}-\beta_{a}\right\vert ^{2}\mid{\mathcal{D}_{h+1}^{N}}\right]
\mathbb{E}_{X\sim\mu_{h}}\left[  |\boldsymbol{\gamma}_{K}(X)|^{2}\right]  \\
\leq\sum_{k=1}^{K}\mathbb{E}_{\mathsf{P}}\left[  \sup_{a\in\mathsf{A}}\left(
\beta_{N,a,k}-\beta_{a,k}\right)  ^{2}\mid{\mathcal{D}_{h+1}^{N}}\right]
\mathbb{E}_{X\sim\mu_{h}}\left[  |\boldsymbol{\gamma}_{K}(X)|^{2}\right]  ,
\end{multline*}
where according to Proposition~\ref{prop:unif-exp}, (component wise applied to
the vector function $f(x,a,\varepsilon)=V_{h+1,N}(\mathcal{K}_{h+1}%
(x,a,\varepsilon))\Sigma^{-1}\boldsymbol{\gamma}_{K}(x)$ with $p=2$, see
\eqref{eq:flip}) one has for $k$ $=$ $1,\ldots,K,$
\begin{equation}
\mathbb{E}_{\mathsf{P}}\left[  \sup_{a\in\mathsf{A}}\left(  \beta
_{N,a,k}-\beta_{a,k}\right)  ^{2}\mid{\mathcal{D}_{h+1}^{N}}\right]
\lesssim\frac{(L_{V,K}L_{\mathcal{K}}I_{\mathcal{D}}(\mathsf{A})+L_{V,K}%
L_{\mathcal{K}}\mathsf{D}(\mathsf{A})+V_{\max}^{\star})^{2}\Lambda_{K}^{2}}%
{N}.\label{rhoK}%
\end{equation}
Due to the very structure of $V_{h+1,N}$ (see (\ref{lowconst})), we further
have
\begin{equation}
\mathbb{E}_{X\sim\mu_{h}}\left[  \sup_{a\in\mathsf{A}}(\beta_{a}^{\top
}\boldsymbol{\gamma}(X)-P_{h+1}^{a}V_{h+1,N}(X))^{2}\mid{\mathcal{D}_{h+1}%
^{N}}\right]  \leq\mathcal{R}_{K,h}^{2},\label{rkh}%
\end{equation}
and then with (\ref{rhoK}) and (\ref{rkh}) we have the estimate
\begin{multline}
\mathbb{E}_{\mu_{h}\otimes\mathsf{P}}\left[  \sup_{a\in\mathsf{A}%
}(\widetilde{P}_{h+1,N}^{a}V_{h+1,N}(X)-P_{h+1}^{a}V_{h+1,N}(X))^{2}%
\mid{\mathcal{D}_{h+1}^{N}}\right]  ^{1/2}\leq\label{one1}\\
\mathbb{E}_{\mu_{h}\otimes\mathsf{P}}\left[  \sup_{a\in\mathsf{A}}(\beta
_{N,a}^{\top}\boldsymbol{\gamma}_{K}(X)-P_{h+1}^{a}V_{h+1,N}(X))^{2}%
\mid{\mathcal{D}_{h+1}^{N}}\right]  ^{1/2}\leq\\
\mathbb{E}_{\mu_{h}\otimes\mathsf{P}}\left[  \sup_{a\in\mathsf{A}}(\beta
_{N,a}^{\top}\boldsymbol{\gamma}_{K}(X)-\beta_{a}^{\top}\boldsymbol{\gamma
}_{K}(X))^{2}\mid{\mathcal{D}_{h+1}^{N}}\right]  ^{1/2}\\
+\mathbb{E}_{X\sim\mu_{h}}\left[  \sup_{a\in\mathsf{A}}(\beta_{a}^{\top
}\boldsymbol{\gamma}_{K}(X)-P_{h+1}^{a}V_{h+1,N}(X))^{2}\mid{\mathcal{D}%
_{h+1}^{N}}\right]  ^{1/2}\\
\leq\varrho_{\gamma,K}\Lambda_{K}(L_{V,K}L_{\mathcal{K}}I_{\mathcal{D}%
}(\mathsf{A})+L_{V,K}L_{\mathcal{K}}\mathsf{D}(\mathsf{A})+V_{\max}^{\star
})\sqrt{\frac{K}{N}}+\mathcal{R}_{K,h}.
\end{multline}
Since the right-hand-side of (\ref{one1}) is deterministic, the conditioning
on ${\mathcal{D}_{h+1}^{N}}$ may be dropped and we obtain%
\begin{multline}
\mathbb{E}_{\mu_{h}\otimes\mathsf{P}}\left[  \sup_{a\in\mathsf{A}%
}(\widetilde{P}_{h+1,N}^{a}V_{h+1,N}(X)-P_{h+1}^{a}V_{h+1,N}(X))^{2}\right]
^{1/2}\leq\label{one}\\
\leq\varrho_{\gamma,K}\Lambda_{K}(L_{V,K}L_{\mathcal{K}}I_{\mathcal{D}%
}(\mathsf{A})+L_{V,K}L_{\mathcal{K}}\mathsf{D}(\mathsf{A})+V_{\max}^{\star
})\sqrt{\frac{K}{N}}+\mathcal{R}_{K,h}.
\end{multline}
\par
\textit{Multi step analysis:} Let us denote for $h\in[H[,$
\begin{equation}
\Delta_{h,N}^{a}(x):=\widetilde{P}_{h+1,N}^{a}V_{h+1,N}(x)-P_{h+1}%
^{a}V_{h+1,N}(x)\text{ \ \ and \ \ }\Delta_{h}(x):=\sup_{a\in\mathsf{A}%
}|\Delta_{h,N}^{a}(x)|.\label{defeps}%
\end{equation}
Note that
$$P_{h+1}^{\pi_{h}}P_{h^{\prime}+1}%
^{\pi_{h^{\prime}}}(dx^{\prime\prime}|x) = \int_{\mathsf{S}}P_{h+1}^{\pi_{h}%
}(dx^{\prime}|x)P_{h^{\prime}+1}^{\pi_{h^{\prime}}}(dx^{\prime\prime
}|x^{\prime}).$$ We then have%
\begin{align}
V_{h}^{\star}(x)-V_{h,N}(x) &  =\sup_{a\in\mathsf{A}}\left\{  R_{h}%
(x,a)+P_{h+1}^{a}V_{h+1}^{\star}(x)\right\}  -\sup_{a\in\mathsf{A}}\left\{
R_{h}(x,a)+\widetilde{P}_{h+1,N}^{a}V_{h+1,N}(x)\right\}  \nonumber\\
&  =R_{h}(x,\pi_{h}^{\star}(x))+\int V_{h+1}^{\star}(x^{\prime})P_{h+1}%
(dx^{\prime}|x,\pi_{h}^{\star}(x))\nonumber\\
&  -\sup_{a\in\mathsf{A}}\left\{  R_{h}(x,a)+\widetilde{P}_{h+1,N}%
^{a}V_{h+1,N}(x)\right\}  \nonumber\\
&  \leq\int\left(  V_{h+1}^{\star}-V_{h+1,N}\right)  (x^{\prime}%
)P_{h+1}(dx^{\prime}|x,\pi_{h}^{\star}(x))\nonumber\\
&  +\sup_{a\in\mathsf{A}}\left\{  R_{h}(x,a)+P_{h+1}^{a}V_{h+1,N}(x)\right\}
-\sup_{a\in\mathsf{A}}\left\{  R_{h}(x,a)+\widetilde{P}_{h+1,N}^{a}%
V_{h+1,N}(x)\right\}  \nonumber\\
&  \leq P_{h+1}^{\pi_{h}^{\star}}\left(  V_{h+1}^{\star}-V_{h+1,N}\right)
(x)+\text{\ }\Delta_{h}(x)\label{vup4}%
\end{align}
and analogously,
\begin{equation}
V_{h}^{\star}(x)-V_{h,N}(x)\geq P_{h+1}^{\pi_{h,N}}[V_{h+1}^{\star}%
-V_{h+1,N}](x)-\Delta_{h}(x).\label{vlow4}%
\end{equation}
By iterating (\ref{vup4}) and (\ref{vlow4}) upwards, and using that
$V_{H,N}=V_{H}^{\star},$ we obtain, respectively,%
\begin{align*}
V_{h}^{\star}(x)-V_{h,N}(x) &  \leq\sum_{k=1}^{H-h-1}P_{h+1}^{\pi_{h}^{\star}%
}\ldots P_{h+k}^{\pi_{h+k-1}^{\star}}[\Delta_{h+k}](x)+\Delta
_{h}(x),\text{\ \ and}\\
V_{h}^{\star}(x)-V_{h,N}(x) &  \geq-\sum_{k=1}^{H-h-1}P_{h+1}^{\pi_{h,N}%
}\ldots P_{h+k}^{\pi_{h+k-1,N}}[\Delta_{h+k}](x)-\Delta_{h}(x).
\end{align*}
We thus have pointwise,%
\begin{align*}
\left\vert V_{h}^{\star}(x)-V_{h,N}(x)\right\vert  &  \leq\sum_{k=1}%
^{H-h-1}P_{h+1}^{\pi_{h}^{\star}}\ldots P_{h+k}^{\pi_{h+k-1}^{\star}%
}[\Delta_{h+k}](x)\\
&  +\sum_{k=1}^{H-h-1}P_{h+1}^{\pi_{h,N}}\ldots P_{h+k}^{\pi_{h+k-1,N}%
}[\Delta_{h+k}](x)+\Delta_{h}(x)
\end{align*}
which implies
\begin{align*}
\left\Vert V_{h}^{\star}-V_{h,N}\right\Vert _{L^{2}(\mu_{h}\otimes\mathsf{P})} &  \leq
2\sup_{\boldsymbol{\pi}}\sum_{k=1}^{H-h-1}\left\Vert P_{h+1}^{\pi_{h}}\ldots
P_{h+k}^{\pi_{h+k-1}}[\Delta_{h+k}]\right\Vert _{L^{2}(\mu_{h}\otimes\mathsf{P})}+\left\Vert
\Delta_{h}\right\Vert _{L^{2}(\mu_h\otimes\mathsf{P})}.
\end{align*}
Hence we have due to Assumption~\ref{ass:mu},
\[
\left\Vert V_{h}^{\star}-V_{h,N}\right\Vert _{L^{2}(\mu_{h}\otimes\mathsf{P})}\leq2\mathfrak{R}%
^{\max}\sum_{l=h}^{H-1}\left\Vert \Delta_{l}\right\Vert _{L^{2}(\mu_{l}\otimes\mathsf{P})}%
\]
(note that $\mathfrak{R}^{\max}\geq1$), and then, by the definitions
(\ref{defeps}) and the estimate (\ref{one}), the
statement of the theorem follows.

\subsection{Proof of Theorem~\ref{dualuperr}}

For the unique minimizer of (\ref{cK}) one has that
\begin{equation}
{\mathbf{c}}_{K}(x,a):=
\Sigma_{\mathsf{E},K} ^{-1}\,\mathbb{E}\left[  V_{t+1}^{\star}(\mathcal{K}%
_{t+1}(x,a,\varepsilon))\boldsymbol{\psi}_{K}(\varepsilon)\right]  .
\end{equation}
Likewise, the unique minimizer of the problem
$$
\inf_{\mathbf{c}\in\mathbb{R}^K}\mathbb{E}_{\varepsilon\sim\mathcal{P}_{\mathsf{E}}%
}\left[  \left(V_{t+1,N}(\mathcal{K}_{t+1}(x,a,\varepsilon))-\mathbf{c}^{\top}\boldsymbol{\psi}%
_{K}(\varepsilon)\right)^2|\mathcal{D}^{N}_{t+1}\right]
$$
is given by $$\bar{\mathbf{c}}_{K}(x,a):=
\Sigma_{\mathsf{E},K}^{-1}\,\mathbb{E}_{\varepsilon\sim\mathcal{P}_{\mathsf{E}}}\left[ V_{t+1,N}(\mathcal{K}_{t+1}(x,a,\varepsilon))\boldsymbol{\psi}_{K}(\varepsilon)|\mathcal{D}^{N}_{t+1}\right].$$
Now let ${\mathbf{c}}_{K,M}(x,a)$ be the Monte Carlo estimate of  $\bar{\mathbf{c}}_{K}(x,a)$ as constructed
in Section~\ref{dualalg}, see (\ref{cbalke}) and
(\ref{cKN}).
We then have
\begin{multline}\label{ga1}
\mathbb{E}_{\mathcal{P}_{\mathsf{E}}\otimes \mathsf{P}}
\left[  \left\vert \sup_{(x,a)\in\mathsf{S}\times\mathsf{A}%
}({\mathbf{c}}_{K,M}-\bar{\mathbf{c}}_{K})^{\top}(x,a)\boldsymbol{\psi}_{K}%
(\varepsilon)\right\vert ^{2}|\mathcal{D}^{N}_{t+1}\right]  \\
\leq\mathbb{E}_{\mathsf{P}}\left[  \sup_{(x,a)\in\mathsf{S}\times\mathsf{A}}\left\vert
({\mathbf{c}}_{K,M}-\bar{\mathbf{c}}_{K})^{\top}(x,a)\right\vert ^{2}|\mathcal{D}%
^{N}_{t+1}\right]  \mathbb{E}_{\varepsilon\sim\mathcal{P}_{\mathsf{E}}}\left[
|\boldsymbol{\psi}_{K}(\varepsilon)|^{2}\right]  ,
\end{multline}
where according to Proposition~ \ref{prop:unif-exp} (applied componentwise
with $p=2$ to the vector function $f(x,a,\varepsilon)=
V_{t+1,N}(\mathcal{K}_{t+1}(x,a,\varepsilon))
\Sigma_{\mathsf{E},K}^{-1}\boldsymbol{\psi}_{K}(\varepsilon)$, see
\eqref{eq:flip})
\begin{multline}
\mathbb{E}_{\mathsf{P}}\left[ \sup_{(x,a)\in \mathsf{S}\times \mathsf{A}%
}\left\vert ({\mathbf{c}}_{K,M}-\bar{\mathbf{c}}_{K})(x,a)\right\vert ^{2}|%
\mathcal{D}_{t+1}^{N}\right]   \label{ga2} \\
\leq \frac{K(L_{V,K_{\text{pr}}}L_{\mathcal{K}}I_{\mathcal{D}}(\mathsf{S}%
\times \mathsf{A})+L_{V,K_{\text{pr}}}L_{\mathcal{K}}\mathsf{D}(\mathsf{S}%
\times \mathsf{A})+V_{\max }^{\star })^{2}\Lambda _{\mathsf{E},K}^{2}}{M}.
\end{multline}
Since for any pair $(x,a)\in \Xset\times \Aset,$%
\begin{multline*}
\left\vert ({\mathbf{c}}_{K}-\bar{\mathbf{c}}_{K})(x,a)\right\vert
^{2}=\left\vert \mathbb{E}_{\varepsilon\sim\mathcal{P}_{\mathsf{E}}}\left[
\left(  V_{t+1}^{\star}(\mathcal{K}_{t+1}(x,a,\varepsilon))-V_{t+1,N}%
(\mathcal{K}_{t+1}(x,a,\varepsilon)\right)  \Sigma_{\mathsf{E},K}%
^{-1}\boldsymbol{\psi}_{K}(\varepsilon)|\mathcal{D}_{t+1}^{N}\right]
\right\vert ^{2}\\
\leq\int\left\vert V_{t+1}^{\star}(\mathcal{K}_{t+1}(x,a,\varepsilon
))-V_{t+1,N}(\mathcal{K}_{t+1}(x,a,\varepsilon))\right\vert ^{2}%
d\mathcal{P}_{\mathsf{E}}(\varepsilon)\,\int\left\vert \Sigma_{\mathsf{E}%
,K}^{-1}\boldsymbol{\psi}_{K}(\varepsilon)\right\vert ^{2}d\mathcal{P}%
_{\mathsf{E}}(\varepsilon)\\
\leq K\Lambda_{\mathsf{E},K}^{2}\sup_{(x,a)\in\mathsf{S}\times\mathsf{A}%
}\left\Vert \frac{dP_{t+1}(\cdot|x,a)}{d\mu_{t+1}(\cdot)}\right\Vert _{\infty
}\int\left\vert V_{t+1}^{\star}(y)-V_{t+1,N}(y)\right\vert ^{2}\mu_{t+1}(dy),
\end{multline*}
we have
\begin{multline}
\mathbb{E}_{\varepsilon\sim\mathcal{P}_{\mathsf{E}}}\left[  \left\vert
\max_{(x,a)\in\mathsf{S}_{L}\times\mathsf{A}_{L}}({\mathbf{c}}_{K}%
-\bar{\mathbf{c}}_{K})^{\top}(x,a)\boldsymbol{\psi}_{K}(\varepsilon
)\right\vert ^{2}|\mathcal{D}_{t+1}^{N}\right]  \label{ga3}\\
\leq\max_{(x,a)\in\mathsf{S}_{L}\times\mathsf{A}_{L}}\left\vert ({\mathbf{c}%
}_{K}-\bar{\mathbf{c}}_{K})(x,a)\right\vert ^{2}\mathbb{E}_{\varepsilon
\sim\mathcal{P}_{\mathsf{E}}}\left[  |\boldsymbol{\psi}_{K}(\varepsilon
)|^{2}\right]  \\
\leq K\varrho_{\psi,K}^{2}\Lambda_{\mathsf{E},K}^{2}\sup_{(x,a)\in
\mathsf{S}\times\mathsf{A}}\left\Vert \frac{dP_{t+1}(\cdot|x,a)}{d\mu
_{t+1}(\cdot)}\right\Vert _{\infty}\Vert V_{t+1}^{\star}-V_{t+1,N}\Vert
_{L^{2}(\mu_{t+1})}^{2}.
\end{multline}
Next due to \eqref{eq:flip}, we derive for any \(k\in [K],\)
\begin{multline*}
  |c_{k,M}(x,a)-c_{k,M}(x^{\prime},a^{\prime})|\\
  \leq  \frac{1}{M}\sum_{m=1}^{M}|V_{t+1,N}(\mathcal{K}_{t+1}%
(x,a,\widetilde\varepsilon_{m}))-V_{t+1,N}(\mathcal{K}_{t+1}(x^{\prime},a^{\prime
},\widetilde\varepsilon_{m}))||\Sigma_{\mathsf{E},K}  ^{-1}\boldsymbol{\psi}_{K}(\widetilde\varepsilon
_{m})|_{\infty}\\
  \leq L_{V,K_\text{pr}}L_{\mathcal{K}}\Lambda_{\mathsf{E},K}\rho((x,a),(x^{\prime
},a^{\prime}))
\end{multline*}
and so
with $I\left[  \mathbf{c}_{K,M}\right]  :=(I\left[  c_{1,M}\right]
,\ldots,I\left[  c_{K,M}\right]  )^{\top}$ we further have%
\begin{align}
&  \mathbb{E}_{\mathcal{P}_{\mathsf{E}}\otimes\mathsf{P}}\left[
\sup_{(x,a)\in\mathsf{S}\times\mathsf{A}}\left\vert \eta_{t+1,K,M}%
(x,a)-\widetilde{\eta}_{t+1,K,M}(x,a)\right\vert ^{2}|\mathcal{D}_{t+1}%
^{N}\right]  \nonumber\\
&  =\mathbb{E}_{\mathcal{P}_{\mathsf{E}}\otimes\mathsf{P}}\left[
\sup_{(x,a)\in\mathsf{S}\times\mathsf{A}}\left\vert \left(  \mathbf{c}%
_{K,M}-I\left[  \mathbf{c}_{K,M}\right]  \right)  ^{\top}(x,a)\boldsymbol{\psi
}_{K}(\varepsilon_{t+1})\right\vert ^{2}|\mathcal{D}_{t+1}^{N}\right]
\nonumber\\
&  \leq\varrho_{\psi,K}^{2}\mathbb{E}_{\mathsf{P}}\left[  \sup_{(x,a)\in
\mathsf{S}\times\mathsf{A}}\left\vert \left(  \mathbf{c}_{K,M}-I\left[
\mathbf{c}_{K,M}\right]  \right)  (x,a)\right\vert ^{2}|\mathcal{D}_{t+1}%
^{N}\right]  \nonumber\\
&  \leq\varrho_{\psi,K}^{2}\sum_{k=1}^{K}\mathbb{E}_{\mathsf{P}}\left[
\sup_{(x,a)\in\mathsf{S}\times\mathsf{A}}\left(  c_{k,M}-I\left[
c_{k,M}\right]  \right)  ^{2}(x,a)|\mathcal{D}_{t+1}^{N}\right]  \nonumber\\
&  \leq K\varrho_{\psi,K}^{2}L_{V,K_{\text{pr}}}^{2}L_{\mathcal{K}}^{2}%
\Lambda_{\mathsf{E,}K}^{2}\rho_{L}^{2}(\mathsf{S},\mathsf{A}),\label{ga4}%
\end{align}
using \eqref{eq:interp-error}. Finally note that
\[
\eta_{t+1,K}-\widetilde{\eta}_{t+1,K,M}=({\mathbf{c}}_{K}-\bar{\mathbf{c}}%
_{K})^{\top}\boldsymbol{\psi}_{K}+(\bar{\mathbf{c}}_{K}-\mathbf{c}%
_{K,M})^{\top}\boldsymbol{\psi}_{K}+\eta_{t+1,K,M}-\widetilde{\eta}_{t+1,K,M}%
\]
and then the result follows by the triangle inequality, gathering
(\ref{ga1})--(\ref{ga4}), and finally taking the unconditional expectation  $\mathbb{E}_{\mathcal{P}_{\mathsf{E}}\otimes\mathsf{P}}$.

\newpage
\appendix

\section{Some auxiliary notions}\label{notation}
The Orlicz 2-norm of a real valued random variable $\eta$ with respect to the
function $\varphi(x)=e^{x^{2}}-1$, $x\in\mathbb{R}$, is defined by $\Vert
\eta\Vert_{\varphi,2}:=\inf\{t>0:\mathbb{E}\left[  {\exp}\left(  {\eta
^{2}/t^{2}}\right)  \right]  \leq2\}$. We say that $\eta$ is
\emph{sub-Gaussian} if $\Vert\eta\Vert_{\varphi,2}<\infty$. In particular,
this implies that for some constants $C,c>0$,
\[
\mathsf{P}(|\eta|\geq t)\leq2\exp\left(  -\frac{ct^{2}}{\Vert\eta
\Vert_{\varphi,2}^{2}}\right)  \text{\ \ and \ \ }\mathbb{E}[|\eta|^{p}%
]^{1/p}\leq C\sqrt{p}\Vert\eta\Vert_{\varphi,2}\text{ \ \ for all \ \ }p\geq1.
\]
Consider a real valued random process $(X_{t})_{t\in\mathcal{T}}$ on a metric
parameter space $(\mathcal{T},\mathsf{d})$. We say that the process has
\emph{sub-Gaussian increments} if there exists $K\geq0$ such that
\[
\Vert X_{t}-X_{s}\Vert_{\varphi,{2}}\leq K\mathsf{d}(t,s),\quad\forall
t,s\in\mathcal{T}.
\]
Let $(\mathsf{Y},\rho)$ be a metric space and $\mathsf{X\subseteq Y}$. For
$\varepsilon>0$, we denote by $\mathcal{N}(\mathsf{X},\rho,\varepsilon)$ the
covering number of the set $\mathsf{X}$ with respect to the metric $\rho$,
that is, the smallest cardinality of a set (or net) of $\varepsilon$-balls in
the metric $\rho$ that covers $\mathsf{X}$. Then $\log\mathcal{N}%
(\mathsf{X},\rho,\varepsilon)$ is called the metric entropy of $\mathsf{X}$
and
\[
I_{\mathcal{D}}(\mathsf{X}):=\int_{0}^{\mathsf{D}(\mathsf{X})}\sqrt
{\log\mathcal{N}\bigl(\mathsf{X},\rho,u\bigr)}\,du
\]
with $\mathsf{D}(\mathsf{X}):=\operatorname{diam}(\mathsf{X}):=\max
_{x,x^{\prime}\in\mathsf{X}}\rho(x,x^{\prime}),$ is called the Dudley
integral. For example, if $|\mathsf{X}|<\infty$ and $\rho(x,x^{\prime
})=1_{\{x\neq x^{\prime}\}}$ we get $I_{\mathcal{D}}(\mathsf{X})=\sqrt
{\log|\mathsf{X}|}.$

\section{Estimation of mean uniformly in parameter}

The following proposition holds.
\begin{proposition}
\label{prop:unif-exp}
Let \(f\) be a function on \(\mathsf{X}\times \Xi\) such that
\begin{eqnarray}
\left|f(x,\xi)-f(x',\xi)\right|\leq L\rho(x,x')
\end{eqnarray}
with some constant \(L>0.\) Furthermore assume that \(\|f\|_{\infty}\leq F<\infty\) for some \(F>0.\) Let \(\xi_n,\) \(n=1,\ldots,N,\) be i.i.d. sample from a distribution on \( \Xi.\) Then we have
\begin{eqnarray*}
\mathbb{E}^{1/p}\left[\sup_{x\in \mathsf{X} }\left\vert \frac{1}{N}\sum_{n=1}^{N}\left(f(x,\xi_n)-\mathbb{E}f(x,\xi_n)\right)\right\vert ^{p}\right]\lesssim \frac{L I_{\mathcal D} +   (L \mathsf{D}  +  F) \sqrt{p}}{\sqrt{N}},
\end{eqnarray*}
where $\lesssim$ may be interpreted as $\le$ up to a
natural constant.
\end{proposition}
\begin{proof}
Denote
\[
Z(x):=\frac{1}{\sqrt{N}}\sum_{n=1}^{N}\left(f(x,\xi_n)-M_{f}(x)\right)
\]
with $M_{f}(x)=\mathsf{E}[f(x,\xi)],$ that is, $Z(x)$ is
a centered random process on the metric space $(\mathsf{X},\rho)$.
Below we show that the process $Z(x)$ has sub-Gaussian increments.
In order to show it, let us introduce
\[
Z_{n}=f(x,\xi_n)-M_{f}(x)-f(x',\xi_n)+M_{f}(x').
\]
Under our assumptions we get
\[
\|Z_{n}\|_{\psi_{2}}\lesssim L\rho(x,x'),
\]
that is, $Z_{n}$ is subgaussian for any $n=1,\ldots,N.$ Since
\[
\ensuremath{Z(x)-Z(x')=N^{-1/2}\sum_{n=1}^{N}Z_{n}},
\]
is a sum of independent sub-Gaussian r.v, we may apply \cite[Proposition 2.6.1 and Eq. (2.16)]{Vershynin}) to obtain that
$Z(x)$ has sub-Gaussian increments with parameter $K\asymp L$. Fix
some $x_0\in \mathsf{X}.$ By the triangular inequality,
\[
\sup_{x\in\mathsf{X}}|Z(x)|\le\sup_{x,x'\in\mathsf{X}}|Z(x)-Z(x')|+\left|Z(x_0)\right|.
\]
By the Dudley integral inequality, e.g. \cite[Theorem
8.1.6]{Vershynin}, for any $\delta\in(0,1)$,
\[
\sup_{x,x'\in\mathsf{X}}|Z(x)-Z(x')|\lesssim L\bigl[I_{\mathcal{D}}+\mathsf{D}\sqrt{\log(2/\delta)}\bigr]
\]
holds with probability at least $1-\delta$. Again, under our assumptions,
$Z(x_0)$ is a sum of i.i.d. bounded centered random variables with
$\psi_{2}$-norm bounded by $F$. Hence, applying Hoeffding's inequality,
e.g. \cite[Theorem 2.6.2.]{Vershynin}, for any
$\delta\in(0,1)$,
\[
|Z(x_0)|\lesssim F\sqrt{\log(1/\delta)}.
\]
\end{proof}

\bibliography{biblio-rl,stop}
\bibliographystyle{plain}

\end{document}